\documentclass[11pt]{article}
\usepackage{mathrsfs}
\usepackage{hyperref}
\usepackage{amsfonts}
\usepackage{amsmath,amssymb,amsthm,latexsym,amstext}
\usepackage[mathscr]{eucal}
\usepackage{rotate,graphics,epsfig,epstopdf}
\usepackage{float}
\usepackage{color}
\usepackage[title]{appendix}
\usepackage{dsfont}

\newcommand{\be}{\begin{eqnarray}}
\newcommand{\ee}{\end{eqnarray}}
\newcommand{\by}{\begin{eqnarray*}}
\newcommand{\ey}{\end{eqnarray*}}
\newcommand{\bn}{\begin{enumerate}}
\newcommand{\en}{\end{enumerate}}
\newcommand{\bi}{\begin{itemize}}
\newcommand{\ei}{\end{itemize}}

\newtheorem{theorem}{Theorem}[section]
\newtheorem{lemma}{Lemma}[section]

\newtheorem{remark}{Remark}[section]
\newtheorem{defin}{Definition}[section]
\newtheorem{proposition}{Proposition}[section]

\newcommand \mF {\mathcal{F}}
\newcommand \mD {\mathcal{D}}
\newcommand \mO {\mathcal{O}}
\newcommand \mR {\mathcal{R}}
\newcommand \mRn {\mR_n}

\newcommand \fR {\mathfrak{R}}

\newcommand \mC {\mathcal{C}}
\newcommand \mL {\mathcal{L}}
\newcommand \mLn {\mL_n}

\newcommand \pr {\mathbb{P}}
\newcommand \E {\mathbb{E}}
\newcommand \En {\E_n}

\newcommand \R {\mathbb{R}}
\newcommand \N {\mathbb{N}}

\newcommand \la {\lambda}
\newcommand \alp {\alpha}

\newcommand \tet {\theta}
\newcommand \kap {\kappa}

\newcommand \del {\delta}
\newcommand \ome {\omega}
\newcommand \eps {\varepsilon}

\newcommand \hR {\hat R}

\newcommand \tauab {s_{\alp b}}

\newcommand \Sy {S_Y}
\newcommand \Fy {F_Y}
\newcommand \fy {f_Y}
\newcommand \hy {h_Y}
\newcommand \My {M_Y}

\newcommand \sqrtn {\sqrt{n}}
\newcommand \lan {\la_n}
\newcommand \Yn {Y_n}
\newcommand \Xn {X_n}
\newcommand \Xnd {X_{n,D}}
\newcommand \Nn {N_n}
\newcommand \tetn {\tet_n}
\newcommand \cn {c_n}
\newcommand \psin {\psi_n}
\newcommand \opsin {\overline{\psi}_n}
\newcommand \opsidn {\overline{\psi}_{D,n}}
\newcommand \hd {h_D}
\newcommand \hn {h_n}
\newcommand \hdn {h_{D,n}}
\newcommand \rhod {\rho_D}
\newcommand \psid {\psi_D}
\newcommand \psidn {\psi_{D,n}}
\newcommand \rhon {\rho_n}
\newcommand \Rn {R_n}
\newcommand \Rd {R_D}
\newcommand \Rdn {R_{D,n}}
\newcommand \gn {g_n}
\newcommand \kn {k_n}

\newcommand \kapn {\kap_n}

\newcommand \Zd {Z_d}
\newcommand \un {u_n}
\newcommand \elln {\ell_n}

\numberwithin{equation}{section}

\begin{document}

\title{Optimal Reinsurance to Minimize the Probability of Drawdown Under the Mean-Variance Premium Principle: \break Asymptotic Analysis}

\author{Pablo Azcue%
\thanks{Departamento de Matematicas, Universidad Torcuato Di Tella. Av. Figueroa Alcorta 7350 (C1428BIJ) Ciudad de Buenos Aires, Argentina, pazcue@utdt.edu.}
\and Xiaoqing Liang%
\thanks{Corresponding author. Department of Statistics, School of Sciences, Hebei University of Technology, Tianjin 300401, P.\ R.\ China, liangxiaoqing115@hotmail.com. X.\ Liang thanks the National Natural Science Foundation of China (11701139) and the Natural Science Foundation of Hebei Province (A2018202057, A2020202033) and the Research Foundation for Returned Scholars of Hebei Province (C20200102) for financial support.}
\and Nora Muler%
\thanks{Departamento de Matematicas, Universidad Torcuato Di Tella. Av. Figueroa Alcorta 7350 (C1428BIJ) Ciudad de
Buenos Aires, Argentina, nmuler@utdt.edu.}
\and Virginia R. Young%
\thanks{Department of Mathematics, University of Michigan, Ann Arbor, Michigan, 48109, vryoung@umich.edu. V.\ R.\ Young thanks the Cecil J. and Ethel M. Nesbitt Professorship of Actuarial Mathematics for financial support.}
}

\maketitle

\date{}

\centerline{{\bf Abstract}}

\medskip

In this paper, we consider an optimal reinsurance problem to minimize the probability of drawdown for the scaled Cram\'er-Lundberg risk model when the reinsurance premium is computed according to the mean-variance premium principle.  We extend the work of Liang et al.\ \cite{LLY2020} to the case of minimizing the probability of drawdown.  By using the comparison method and the tool of adjustment coefficients, we show that the minimum probability of drawdown for the scaled classical risk model converges to the minimum probability for its diffusion approximation, and the rate of convergence is of order $\mO(n^{-1/2})$.  We further show that using the optimal strategy from the diffusion approximation in the scaled classical risk model is $\mO(n^{-1/2})$-optimal.

\medskip


\medskip


\medskip


\section{Introduction}

Drawdown occurs when the value of a portfolio or company's surplus is lower than a proportion of its historic maximum, which is an important risk metric for fund or corporate managers when evaluating their portfolios or companies, respectively. The frequent occurrence of drawdown implies low financial profit, potential large losses, or even bankruptcy, which may help managers take effective action for the company's operation.

Due to the importance of analyzing the occurrence of drawdown, recently researchers have been considering the problem of minimizing the probability of drawdown.  Specifically, a decision maker chooses an optimal strategy to minimize the probability that the fund or surplus decreases to a fixed proportion, say $\alp \in [0,1)$, of its historic maximum value.  If we set $\alp = 0$, then minimizing the probability of drawdown degenerates to minimizing the probability of ruin.  Angoshtari et al.\ \cite{ABY16a} minimize the probability of drawdown with investment in a Black-Scholes financial market. They find that the strategy for minimizing the probability of drawdown is identical to the one for minimizing the probability of ruin.  Han et al.\ \cite{HLY2020} consider the reinsurance problem of minimizing the probability of drawdown under the mean-variance premium principle and under the diffusion approximation of the classical Cram\'er-Lundberg risk model. They find an explicit form of the optimal reinsurance strategy and show that the optimal reinsurance strategy for minimizing the probability of drawdown coincides with the one for minimizing the probability of ruin, as in Angoshtari et al.\ \cite{ABY16a}.  There is also research related to optimization problems under drawdown from the perspective of individual investors; see, for example, Grossman and Zhou \cite{GZ1993}, Cvitani\'c and Karatzas \cite{CK1995}, Elie and Touzi \cite{ET2008}, Chen et al.\ \cite{CLLL15}, and Angoshtari et al.\ \cite{ABY16b}.

There are many papers that employ the diffusion approximation in the actuarial science literature. By comparison, research on optimization problems with jump-diffusion models or classical Cram\'er-Lundberg (CL) models are less.  The reason is that explicit solutions for the latter models are difficult to derive.  Some researchers use probabilistic techniques to verify the convergent relationship between the probability of ruin for the scaled model and that of its corresponding diffusion approximation; see, for example, Iglehart \cite{I1969}, Grandell \cite{G1977}, Asmussen \cite{A1984}, and B\"auerle \cite{B2004}.  More recently, Cohen and Young \cite{CY2020} use the comparison method from differential equations to find upper and lower bounds of the probability of ruin and prove that the probability of ruin for the scaled CL model converges to the probability of ruin for the corresponding diffusion approximation. Liang et al.\ \cite{LLY2020} consider the problem of minimizing the probability of ruin with reinsurance under the mean-variance premium principle. They show that the minimum probability of ruin is the unique viscosity solution of the corresponding Hamilton-Jacobi-Bellman equation with boundary conditions. They prove that, under appropriate scaling of the CL risk model, the probability of ruin converges to its diffusion approximation.  Liang and Young \cite{LY2021} extend Cohen and Young \cite{CY2020} to the exponential Parisian ruin model.  Cohen and Young \cite{CY2021} prove asymptotic results for the optimal-dividend problem.

In this paper, we analyze the relationship between the probability of drawdown for the scaled classical CL risk model and its corresponding diffusion approximation.  We first prove two ``smooth'' comparison results. Then, we use the comparison results and the analogs of adjustment coefficients to prove that the minimum probability of drawdown for the scaled classical risk model converges to the minimum probability for its diffusion approximation uniformly, with rate of convergence of order $\mO(n^{-1/2})$.  To justify using the diffusion approximation in this optimization problem, we show that using the optimal retention strategy from the diffusion approximation is $\mO(n^{-1/2})$-optimal in the scaled classical CL risk model. 

The rest of the paper is organized as follows.  In Section \ref{sec:CL}, we present the classical CL model for the insurer's surplus, we describe the reinsurance market, and we define the minimum probability of drawdown for the classical risk process.  Then, in Section \ref{sec:compare}, we provide two comparison theorems that form the backbone of the proofs in Section \ref{sec:bounds}.  In Section \ref{sec:scale_CL}, we define the scaled model and discuss how to extend the comparison theorems to this model.  Section \ref{sec:diff} provides an explicit expression for the minimum probability of drawdown for the diffusion approximation of the scaled classical risk model, and Section \ref{sec:outline} gives an outline of the remainder of the paper.

In Sections \ref{sec:adj_coeff_CL} and \ref{sec:adj_coeff_diff}, we define analogs of the adjustment coefficient for the scaled classical risk model and for its diffusion approximation, respectively, and in Section \ref{sec:relation}, we prove two lemmas that show how these adjustment coefficients are related.  In Section \ref{sec:bounds}, we use the comparison lemma from Section \ref{sec:compare} to modify the minimum probability of drawdown for the diffusion approximation by functions of order $\mO(n^{-1/2})$ to obtain upper and lower bounds of the corresponding probability of drawdown for the scaled classical risk model.  In Section \ref{sec:psin_lim}, we use these bounds to prove that the minimum probability of drawdown for the scaled classical risk model converges to that of its diffusion approximation.  We show that the rate of convergence is of order $\mO(n^{-1/2})$, uniformly with respect to the surplus.  Finally, in Section \ref{sec:Rd_opt}, we prove the main result of our paper, namely, that if the insurer uses the optimal strategy from the diffusion approximation, then the resulting probability of drawdown is $\mO(n^{-1/2})$-optimal.  This result thereby justifies using the diffusion approximation when minimizing the probability of drawdown.

\section{Scaled Cram\'er-Lundberg model and its diffusion approximation}

\subsection{Cram\'er-Lundberg model and the probability of drawdown}\label{sec:CL}

In this section, we describe the reinsurance market available to the insurance company, and we formulate the problem of minimizing the probability of drawdown.  Assume that all random processes exist on the filtered probability space $\big( \Omega, \mathcal{F}, \mathbb{F} = \{ \mF_t \}_{t \ge 0}, \mathbb{P} \big)$.

We model the insurer's claim process $C = \{C_t\}_{t \ge 0}$ according to a compound Poisson process, namely,
\begin{equation}
\label{eq:CPP}
C_t = \sum_{i = 1}^{N_t} Y_i,
\end{equation}
in which the claim severities $Y_1, Y_2, \dots$ are independent and identically distributed according to a common cumulative distribution function $\Fy$, with $\Fy(0) = 0$, and in which the claim frequency $N = \{ N_t \}_{t \ge 0}$ follows a Poisson process with parameter $\la > 0$.  Let $\Sy = 1 - \Fy$ denote the survival function of the severity random variable $Y$, and assume that $Y$'s moment generating function $\My$ is finite in a neighborhood of $0$.  Also, assume that the insurer receives premium payable continuously at a rate $c > \la \E Y$, and assume that the Poisson process $N$ is independent of the claim severity process $\{Y_i\}_{i \in \N}$.

\begin{remark}\label{rem:mgf}
From Lemma $2.3.1$ in Rolski et al.\ {\rm \cite{RSST1999}}, we know that, because $\My(t_0) < \infty$ for some $t_0 > 0$, then there exists $b > 0$ such that 
\begin{equation}\label{eq:exp_Sy}
\Sy(y) \le b e^{-t_0 y}, \qquad \forall y \ge 0.
\end{equation}
In other words, if $\My$ is finite in a neighborhood of $0$, then $\Sy$ has an exponentially decreasing right tail.  Conversely, if \eqref{eq:exp_Sy} holds, that is, if $\Sy$ has an exponentially decreasing right tail, then $\My(t) < \infty$ for all $t < t_0$.  \qed
\end{remark}

We assume that the insurer can buy per-loss reinsurance, with a continuously payable premium computed according to the so-called \textit{mean-variance} premium principle, which combines the expected-value and variance premium principles, with risk loadings $\tet$ and $\eta$, respectively.  Specifically, if $R_t(\omega, y)$ represents the retained claim at time $t \ge 0$, as a function of the (possible) claim $Y = y$ at that time and state of the world $\omega \in \Omega$, then reinsurance indemnifies the insurer by the amount $y - R_t(\omega, y)$ if there is a claim $y$ at time $t \ge 0$ and $\omega \in \Omega$, and the time-$t$ premium rate is given by
\begin{equation}\label{eq:reins_prem}
(1 + \tet) \la \E(Y - R_t) + \dfrac{\eta}{2} \, \la \E((Y - R_t)^2).
\end{equation}
Assume that
\begin{equation}\label{eq:c_not_huge}
c < (1 + \tet) \la \E Y + \dfrac{\eta}{2} \, \la \E(Y^2);
\end{equation}
in words, the insurer's premium income is not sufficient to buy full reinsurance, and let $\kap$ denote the positive difference
\begin{equation}\label{eq:kap}
\kap = (1 + \tet) \la \E Y + \dfrac{\eta}{2} \, \la \E(Y^2) - c.
\end{equation}

\begin{defin}\label{def:adm}
A strategy $\mR = \{ R_t \}_{t \ge 0}$ is an {\rm admissible retention strategy} if it satisfies the following properties:
\begin{itemize}

\item[$(i)$] $\mR$ is predictable; that is, the function $(\omega, y) \mapsto R_t(\omega, y)$ is $\mF_{t^-} \times \mathcal{B}(\R_+)$-measurable for every $t \ge 0$, in which $\mathcal{B}(\R_+)$ denotes the Borel $\sigma$-algebra on $\R_+$.


\item[$(ii)$] $0 \le R_t(\omega, y) \le y$, for all $t \ge 0$ and $\omega \in \Omega$.

\item[$(iii)$] The net premium of the controlled surplus is greater than the expected rate of claim payment, that is, 
\begin{equation}\label{eq:net_prof}
c - (1 + \tet) \la \E(Y - R_t) - \dfrac{\eta}{2} \, \la \E((Y - R_t)^2) > \la \E R_t ,
\end{equation}
with probability one for all $t \ge 0$.  Hald and Schmidli {\rm \cite{HS2004}} and Liang and Guo {\rm \cite{LG2007}}, among others,
refer to inequality \eqref{eq:net_prof} as the {\rm net-profit condition}.

\item[$(iv)$] $\mR$ is progressively measurable; that is, the function $(\omega, y, t) \mapsto R_t(\omega, y)$ is $\mF_t \times \mathcal{B}(\R_+) \times \mathcal{B}(\R_+)$-measurable.
\end{itemize}
When initial surplus equals $x$, denote the set of admissible strategies by $\fR_x$.

Moreover, a function $R: \R_+ \to \R_+$ is an {\rm admissible retention function} if
\begin{itemize}
\item[$(i)$] The mapping $y \mapsto R(y)$ is $\mathcal{B}(\R_+)$-measurable.

\item[$(ii)$] $0 \le R(y) \le y$, for all $y \ge 0$.

\item[$(iii)$] $c - (1 + \tet) \la \E(Y - R) - \dfrac{\eta}{2} \, \la \E((Y - R)^2) > \la \E R$.
\end{itemize}
In other words, a function $R$ is an admissible retention function if the constant strategy $\mR = \{R_t \equiv R \}$ is an admissible retention strategy.  \qed
\end{defin}
Given a retention strategy $\mR \in \fR_x$, the insurer's surplus follows the dynamics
\begin{align}
\label{eq:X}
dX^\mR_t &= \left( c - (1 + \tet) \la \E(Y - R_t) - \dfrac{\eta}{2} \, \la \E((Y - R_t)^2) \right)dt - R_t dN_t \notag \\
&= \left( - \kap + \la \left( (1 + \tet) \E R_t + \eta \E(Y R_t) - \dfrac{\eta}{2} \, \E(R^2_t) \right) \right)dt - R_t dN_t,
\end{align}
with $X^\mR_{0^-} = x \ge 0$.\footnote{We assume that no claim occurs at time $0$, which is true with probability 1, so we also have $X^\mR_{0} = x \ge 0$.}  Define the corresponding maximum surplus process $M^\mR = \big\{ M^\mR_t \big\}_{t \ge 0}$ by
\begin{equation}\label{eq:M}
M^\mR_t = \max \bigg\{ \sup \limits_{0 \le s \le t} X^\mR_s, \, M^\mR_{0^-} \bigg\},
\end{equation}
with $M^\mR_0 = m \ge x$.  We allow the surplus process to have a financial past, as embodied by the term $M^\mR_{0^-}$ in \eqref{eq:M}.  Drawdown is the time when the surplus process drops below the proportion $\alp \in [0,1)$ of its maximum value, that is, at the hitting time $\tau^\mR_\alp$ given by
\begin{equation}\label{eq:tau_alp}
\tau^\mR_\alp = \inf \Big\{t \ge 0:\ X^\mR_t < \alp M^\mR_t \Big \}.
\end{equation}
If $\alp = 0$, then the time or event of drawdown is the same as that for ruin under the ruin level $0$.

The goal of the insurer is to minimize its probability of drawdown by purchasing per-loss reinsurance.  The corresponding minimum probability of drawdown $\psi$ is defined by
\begin{equation}\label{eq:phi}
\psi(x, m) = \inf \limits_{\mR \in \fR_x} \pr^{x, m} \big(\tau^\mR_\alp < \infty \big) = \inf \limits_{\mR \in \fR_x} \E^{x, m} \big({\bf 1}_{\{\tau^\mR_\alp < \infty\}} \big),
\end{equation}
in which $\pr^{x, m}$ and $\E^{x,m}$ denote the probability and expectation, respectively, conditional on $X_0 = x$ and $M_0 = m$.   Note that, if $x < \alp m$, then $\psi(x, m) = 1$.  It remains for us to study the minimum probability of drawdown $\psi$ on the domain
\begin{equation} \label{eq:mD}
\mD = \big\{(x, m) \in \R_+^2: \alp m \le x \le m \big\}.
\end{equation}
In the following proposition, we prove some interesting properties of $\psi$.

\begin{proposition}
$\psi(x, \cdot)$ is nondecreasing on $\mD$, and $\psi(\cdot, m)$ is nonincreasing and Lipschitz on $\mD$.  Moreover,
\begin{equation}\label{eq:psi_lim}
\lim_{m \to \infty} \psi(m, m) = 0.
\end{equation}
\end{proposition}

\begin{proof}
For a fixed value of $x$, if $(x \le) \; m_1 < m_2$, then the ruin levels are ordered, that is, $\alp m_1 < \alp m_2$, which implies that the probability of drawdown under ruin level $\alp m_1$ is less than or equal to the probability of drawdown under $\alp m_2$.  In other words, $\psi(x, \cdot)$ is nondecreasing on $\mD$.

For a fixed value of $m$, if $\alp m \le x_1 < x_2 \le m$, then clearly the probability of drawdown when surplus equals $x_1$ is greater than or equal to the probability of drawdown when surplus equals $x_2$.  In other words, $\psi(\cdot, m)$ is nonincreasing on $\mD$.  It remains to show that  $\psi(\cdot, m)$ is Lipschitz on $\mD$.

For a fixed value of $m$, suppose $\alp m \le x_1 < x_2 \le m$.  Choose an admissible retention function $\hat{R}$, and form the constant admissible retention strategy $\hat\mR = \{\hat{R}\}_{t\ge 0} \in \fR_{x_1}$.  Given $\eps > 0$, let $\bar{\mR}_1 = \{R^1_t\}_{t \ge 0} \in \fR_{x_2}$ be such that
\begin{align*}
\psi(x_2, m; \bar{\mR}_1) \le \psi(x_2,m) + \eps,
\end{align*}
in which $\psi(\cdot, \, \cdot; \mR)$ denotes the probability of drawdown when the insurer follows strategy $\mR$.

Now, define $\bar{\mR}=\{R_t\}_{t \ge 0} \in \fR_{x_1}$ as follows:
\[
R_t =
\begin{cases}
\hat{R}, &\quad t\le \tau_{x_2} := \inf \{t\ge0: X^{\bar{\mR}}_t=x_2\}, \\
R^{1}_{t- \tau_{x_2}}, &\quad t>\tau_{x_2}.
\end{cases}
\]
If no claims occur, then the process $X^{\bar{\mR}}$ with initial surplus $x_1$ reaches $x_2$ at time $h = (x_2 - x_1)/p_{\hat{R}}$, in which
\begin{equation}\label{eq:pR}
p_R = c - (1 + \tet) \la \E(Y - R) - \frac{\eta}{2} \la \E((Y - R)^2).
\end{equation}
Note that condition (iii) in Definition \ref{def:adm} implies $p_R > \la \E R$ for any admissible retention function $R$.  Thus, if we let $\tau_1$ denote the time of the first claim when initial surplus equals $x_1$, we have
\begin{align*}
1 - \psi(x_1, m) &\ge 1- \psi(x_1, m; \bar{\mR}) \notag \\
&\ge \big(1 - \psi(x_2, m; \bar{\mR}) \big) \pr(\tau_1 > h) \notag \\
&= \big(1 - \psi(x_2, m; \bar{\mR}) \big) e^{-\la h} \notag \\
&\ge (1-\psi(x_2,m)- \eps) e^{- \la h}.
\end{align*}
Because $\eps > 0$ is arbitrary, we deduce
\begin{equation}\label{ineq1}
1 - \psi(x_1, m) \ge (1-\psi(x_2,m)) e^{- \la h}.
\end{equation}
From inequality \eqref{ineq1} and $\psi(\cdot, m)$ nonincreasing, we obtain
\begin{align*}
0 &\le \psi(x_1,m) - \psi(x_2,m) \le (1 - \psi(x_2,m))(1-e^{-\la h})  \\
& \le 1 - e^{-\la h} \le \la h = \dfrac{\la}{p_{\hat R}} \, (x_2 - x_1).
\end{align*}
Because $\hat{R}$ is arbitrary such that $p_{\hat R} > \la \E \hat R$, it follows that
\[
\big| \psi(x_1,m) - \psi(x_2,m) \big| \le \dfrac{\la}{\sup_{R} p_{R}} \, |x_2 - x_1|.
\]
Thus, $\psi$ is Lipschitz in $x$.

It remains to prove the limit in \eqref{eq:psi_lim}.  To that end, for any value $(x, m) \in \mD$, let $R$ be an admissible retention function, and let $\mR = \{ R\}_{t \ge 0} \in \fR$ be the corresponding constant retention strategy.  We wish to prove that
\begin{equation}\label{ineq:psi_mR}
\psi(x, m; \mR) \le e^{-J(x - \alp m)},
\end{equation}
in which $J$ is the positive solution of the following equation:
\begin{equation}\label{eq:J}
\la \big(\E e^{J R} - 1 - J \E R \big) = J \left[ - \kap + \la \left( \tet \E R + \eta \E(YR) - \dfrac{\eta}{2} \, \E(R^2) \right) \right].
\end{equation}
As an aside, to see that \eqref{eq:J} has a unique positive solution, note that condition (iii) of Definition \ref{def:adm} implies the coefficient of $J$ on the right side is positive. The two sides equal $0$ when $J = 0$, and the rates of growth with respect to $J$ when $J = 0$ of the left and right sides equal, respectively, $0$ and the positive coefficient in square brackets.  Thus, because the left side grows exponentially and the right grows linearly with $J$, there is a unique positive solution of \eqref{eq:J}.

Next, let $\psi_k(\cdot, \, \cdot; \mR)$ denote the probability of drawdown (under the strategy $\mR$) at or before the $k^{th}$ claim. For all $(x, m) \in \mD$, $\psi_k(x, m; \mR)$ increases with respect to $k$, and $\lim \limits_{k \to \infty} \psi_k(x, m; \mR) = \psi(x, m; \mR)$.  Indeed, if we let $T_k$ denote the time of the $k^{th}$ claim, then $\psi_k(x, m; \mR) = \pr(\tau_\alp^\mR \le T_k)$, and $\lim \limits_{k \to \infty} T_k = \infty$, which give us
\[
\lim_{k \to \infty} \psi_k(x, m; \mR) = \lim_{k \to \infty} \pr(\tau_\alp^\mR \le T_k) = \pr \big( \cup_{k=1}^\infty \{\tau_\alp^\mR \le T_k \} \big) = \pr(\tau_\alp^\mR < \infty) = \psi(x, m; \mR).
\]
Thus, to prove \eqref{ineq:psi_mR}, it is enough to prove
\begin{equation}\label{ineq:psik_mR}
\psi_k(x, m; \mR) \le e^{-J(x - \alp m)},
\end{equation}
for all $k \in \N$.  By a recursive argument, $\psi_k$ solves
\begin{equation}\label{eq:psik_recursion}
\psi_k(x, m; \mR) = \int_0^\infty \int_0^\infty \psi_{k-1}(x + p_R t - R(y), m; \mR) \la e^{-\la t} d\Fy(y) dt,
\end{equation}
for $k \in \N$, in which
\begin{equation}
\psi_0(x, m; \mR) =
\begin{cases}
0, &\quad x \ge \alp m, \\
1, &\quad x < \alp m.
\end{cases}
\end{equation}
For $(x, m) \in \mD$, $\psi_0(x, m; \mR) = 0 \le e^{-J(x - \alp m)}$.  Assume \eqref{ineq:psik_mR} holds for some $k - 1 = 0, 1, 2, \dots$; then,
\begin{align*}
\psi_k(x, m; \mR) &\le \int_0^\infty \int_0^\infty e^{-J(x + p_R t - R(y) - \alp m)} \la e^{-\la t} d\Fy(y) dt \\
&= e^{-J(x - \alp m)} \int_0^\infty \la e^{J R(y)} d\Fy(y) \cdot \int_0^\infty e^{- (\la + Jp_R)t} dt \\
&= e^{-J(x - \alp m)} (\la + J p_R) \cdot \dfrac{1}{\la + J p_R} = e^{-J(x - \alp m)},
\end{align*}
in which the second equality follows from \eqref{eq:J}.  Thus, inequality \eqref{ineq:psik_mR} follows by recursion, which implies inequality \eqref{ineq:psi_mR}.  From the minimality of $\psi$ and from \eqref{ineq:psi_mR}, we obtain
\[
0 \le \lim_{m \to \infty} \psi(m, m) \le \lim_{m \to \infty} \psi(m, m; \mR) \le \lim_{m \to \infty} e^{-J(1 - \alp)m} = 0,
\]
which proves the limit in \eqref{eq:psi_lim}.
\end{proof}

\subsection{Comparison theorems}\label{sec:compare}

For the classical risk model, we cannot find an explicit expression for the minimum probability of drawdown $\psi$, so in Section \ref{sec:diff_approx}, we prove that the minimum probability of drawdown for the diffusion approximation of $X^\mR$ approximates $\psi$.     To that end, in this section, we prove two comparison theorems which we use in Section \ref{sec:diff_approx}.

We present the following lemma, which we will use in the proofs of Theorems \ref{thm:sub} and \ref{thm:super} below.

\begin{lemma}\label{lem:tauab_fin}
For $(x, m) \in \mD$, fix an admissible retention strategy $\mR \in \fR_x$.  Define $s_b = \inf\{t \ge 0: X^\mR_t \ge b\}$ with $X^\mR_0 = x$, and define $\tauab = \tau_\alp \wedge s_b$.\footnote{These stopping times depend on $\mR$ via $X^\mR$, but for simplicity of notation in the remainder of this section, we suppress the superscript $\mR$ on $\tau_\alp$, $s_b$, and $\tauab$.}  Then, $\pr^{x, m}(\tauab < \infty) = 1$.
\end{lemma}

\begin{proof}
Proposition 2.2 in Azcue and Muler \cite{AM2014} states: with probability one, either ruin occurs in finite time or $X^\mR_t$ diverges to infinity as $t$ goes to infinity.  Because $\alp M_t \ge 0$ and because $b < \infty$, it follows that $\pr^{x, m}(\tauab < \infty) = 1$.
\end{proof}

In the first theorem, we show that any smooth subsolution of our problem is less than or equal to the value function.  Before stating the theorem, we introduce some notation.  For an admissible retention function $R$, define the operator $\mL^R$ on $\mC^{1,1}(\mD)$ as follows:  for $u \in \mC^{1,1}(\mD)$,
\begin{align}
\label{eq:mL}
\mL^R u(x, m) &= - \kap u_x(x, m) \notag \\
&\quad + \la \left[ \left( (1 + \tet) \E R + \eta \E(YR) - \dfrac{\eta}{2} \, \E(R^2) \right) u_x(x,m) + \E u(x - R, m) - u(x, m) \right],
\end{align}
in which we extend $u$ by defining $u(x, m) = 1$ for $x < \alp m$.  Note that the HJB equation for $\psi$ is $\inf_R \mL^R \psi(x, m) = 0$.

\begin{theorem}[Subsolution]\label{thm:sub}
Suppose $u \in \mC^{1,1}(\mD)$ is a bounded function that satisfies the following conditions:
\begin{enumerate}
\item[$(i)$] $\lim \limits_{m \to \infty} u(m, m) \le 0$. 
\item[$(ii)$] $u(x, m)$ is defined for all $x \le m$ and $m \ge 0$, with $u(x, m) = 1$ for all $x < \alp m$.
\item[$(iii)$] $u_m(m, m) \ge 0$ for all $m \ge 0$.
\item[$(iv)$] $\mL^R u(x, m) \ge 0$ for all $(x, m) \in \mD$ and for all admissible retention functions $R$.
\end{enumerate}

\noindent
Then, $u \le \psi$ on $\mD$.
\end{theorem}

\begin{proof}
Assume that $u$ satisfies the conditions specified in the statement of this theorem, and fix an admissible retention strategy $\mR$.

For a fixed value of $m \ge 0$, let $b > m$, and define $s_b$ and $\tauab$ as in Lemma \ref{lem:tauab_fin}.  By applying It\^{o}'s formula to $u(x, m)$, we have
\begin{align}\label{eq:sub1}
u \big(X^\mR_{\tauab}, M^\mR_{\tauab} \big) &= u(x, m) + \int_0^{\tauab} \left( - \kap + \la \left( (1 + \tet) \E R_t + \eta \E(Y R_t) - \dfrac{\eta}{2} \, \E(R^2_t) \right) \right) u_x \big(X^\mR_{t^-}, M^\mR_t \big) \, dt \notag \\
&\quad + \int_0^{\tauab} \left( \E u \big(X^\mR_{t^-} - R_t, M^\mR_t \big) - u \big(X^\mR_{t^-}, M^\mR_t \big) \right) dN_t + \int_0^{\tauab} u_m \big(X^\mR_{t^-}, M^\mR_t \big) \, dM^\mR_t \notag \\
&= u(x, m) + \int_0^{\tauab} \mL^{R_t} u \big(X^\mR_{t^-}, M^\mR_t \big) \, dt \notag \\
&\quad + \int_0^{\tauab} \left( \E u \big(X^\mR_{t^-} - R_t, M^\mR_t \big) - u \big(X^\mR_{t^-}, M^\mR_t \big) \right) d(N_t - \la t) \notag \\
&\quad + \int_0^{\tauab} u_m \big(X^\mR_{t^-}, M^\mR_t \big) \, dM^\mR_t.
\end{align}
The first integral in \eqref{eq:sub1} is non-negative because of condition (iv) of the theorem.  The expectation of the second integral equals $0$ because $u$ is bounded.  The third integral is non-negative almost surely because $dM^\mR_t$ is non-zero only when $M^\mR_t = X^\mR_t$ and $u_m(m, m) \ge 0$ by condition (iii).  Here, we also used the fact that $M^\mR$ is non-decreasing; therefore, the first variation process associated with it is finite almost surely, and we conclude that the cross variation of $M^\mR$ and $X^\mR$ is zero almost surely.  Thus, by taking expectations in \eqref{eq:sub1}, we have
\begin{equation}\label{eq:sub2}
\E^{x, m} \left[u \big(X^\mR_{\tauab}, M^\mR_{\tauab} \big) \right] \ge u(x, m).
\end{equation}

Because $\tauab < \infty$ with probability 1 and because $b > m$, it follows from the extension of $u$ to $u(x, m) = 1$ for all $x < \alp m$ and from inequality \eqref{eq:sub2} that
\begin{align}\label{eq:sub3}
u(x, m) &\le \pr^{x, m}(\tau_\alp < s_b) \cdot 1 + \pr^{x, m}(s_b < \tau_\alp) \cdot u(b, b) \notag \\
&\le \pr^{x, m}(\tau_\alp < \infty) + \pr^{x, m}(s_b < \tau_\alp) \cdot u(b, b).
\end{align}
In \eqref{eq:sub3}, we purposefully omit $\pr^{x, m}(\tau_\alp = s_b)$ because that probability equals $0$, which follows from $\pr^{x, m}(\tauab < \infty) = 1$ and $\alp m < b$.

By applying the Dominated Convergence Theorem to \eqref{eq:sub3} as we take the limit $b \to \infty$ and by using $\lim_{b \to \infty} u(b, b) \le 0$ from condition (i), we obtain
\begin{equation}\label{eq:sub4}
u(x, m) \le \pr^{x, m}(\tau_\alp < \infty).
\end{equation}
By taking the infimum over admissible strategies, we obtain $u \le \psi$ on $\mD$.
\end{proof}

In the second theorem, we show that any smooth supersolution of our problem is greater than or equal to the value function.

\begin{theorem}[Supersolution]\label{thm:super}
Suppose $v \in \mC^{1,1}(\mD)$ is a bounded function that satisfies the following conditions:
\begin{enumerate}
\item[$(i)$] $\lim \limits_{m \to \infty} v(m, m) \ge 0$. 
\item[$(ii)$] $v(x, m)$ is defined for all $x \le m$ and $m \ge 0$, with $v(x, m) = 1$ for all $x < \alp m$.
\item[$(iii)$] $v_m(m, m) \le 0$ for all $m \ge 0$.
\item[$(iv)$] $\mL^{\hR} v(x, m) \le 0$ for all $(x, m) \in \mD$ and for some admissible retention function $\hR$.
\end{enumerate}

\noindent
Then, $\psi \le v$ on $\mD$.
\end{theorem}

\begin{proof}
Assume that $v$ and $\hR$ satisfy the conditions specified in the statement of this theorem, and let $\hat{\mR}$ denote the constant retention strategy corresponding to $\hR$.

For a fixed value of $m \ge 0$, let $b > m$, and define $s_b$ and $\tauab$ as in Lemma \ref{lem:tauab_fin}.  By applying It\^{o}'s formula to $v(x, m)$, we have
\begin{align}\label{eq:super1}
v \Big(X^{\hat \mR}_{\tauab}, M^{\hat \mR}_{\tauab} \Big) &= v(x, m) + \int_0^{\tauab} \left( - \kap + \la \left( (1 + \tet) \E \hR_t + \eta \E \big( Y \hR_t \big) - \dfrac{\eta}{2} \, \E \big( \hR^2_t \big) \right) \right) v_x \Big(X^{\hat \mR}_{t^-}, M^{\hat \mR}_t \Big) \, dt \notag \\
&\quad + \int_0^{\tauab} \left( \E v \Big(X^{\hat \mR}_{t^-} - \hR_t, M^{\hat \mR}_t \Big) - v \Big(X^{\hat \mR}_{t^-}, M^{\hat \mR}_t \Big) \right) dN_t + \int_0^{\tauab} v_m \Big(X^{\hat \mR}_{t^-}, M^{\hat \mR}_t \Big) \, dM^{\hat \mR}_t \notag \\
&= v(x, m) + \int_0^{\tauab} \mL^{\hR_t} v \Big(X^{\hat \mR}_{t^-}, M^{\hat \mR}_t \Big) \, dt \notag \\
&\quad + \int_0^{\tauab} \left( \E v \Big(X^{\hat \mR}_{t^-} - \hR_t, M^{\hat \mR}_t \Big) - v \Big(X^{\hat \mR}_{t^-}, M^{\hat \mR}_t \Big) \right) d(N_t - \la t) \notag \\
&\quad + \int_0^{\tauab} v_m \Big(X^{\hat \mR}_{t^-}, M^{\hat \mR}_t \Big) \, dM^{\hat \mR}_t.
\end{align}
The first integral in \eqref{eq:super1} is non-positive because of condition (iv) of the theorem.  The expectation of the second integral equals $0$ because $v$ is bounded.  The third integral is non-positive almost surely because $dM^{\hat \mR}_t$ is non-zero only when $M^{\hat \mR}_t = X^{\hat \mR}_t$ and $v_m(m, m) \le 0$ by condition (iii).  Thus, by taking expectations in \eqref{eq:super1}, we have
\begin{equation}\label{eq:super2}
\E^{x, m} \left[v \Big(X^{\hat \mR}_{\tauab}, M^{\hat \mR}_{\tauab} \Big) \right] \le v(x, m).
\end{equation}

Because $\tauab < \infty$ with probability 1 and because $b > m$, it follows from the extension of $v$ to $v(x, m) = 1$ for all $x < \alp m$ and from inequality \eqref{eq:super2} that
\begin{align}\label{eq:super3}
v(x, m) &\ge \pr^{x, m}(\tau_\alp < s_b) \cdot 1 + \pr^{x, m}(s_b < \tau_\alp) \cdot v(b, b).
\end{align}
Next, we wish to show that
\begin{equation}\label{eq:super3a}
\lim_{b \to \infty} \{\tau_\alp < s_b\} = \{\tau_\alp < \infty\}.
\end{equation}
Because $\{\tau_\alp < s_b\} \subseteq \{\tau_\alp < \infty\}$ for all $b > m$, it follows that $\lim_{b \to \infty} \{\tau_\alp < s_b\} \subseteq \{\tau_\alp < \infty\}$.  To prove the opposite inclusion, suppose $\omega \in \{\tau_\alp < \infty\}$, that is, drawdown occurs in finite time, say, at time $t$.  Thus, $M^{\hat \mR}_t < \infty$, which implies $\omega \in \{\tau_\alp < s_b\}$ for all $b > M^{\hat \mR}_t$, and we have shown that $\{\tau_\alp < \infty\} \subseteq \lim_{b \to \infty} \{ \tau_\alp < s_b \}$, from which \eqref{eq:super3a} follows.

By applying the Dominated Convergence Theorem to \eqref{eq:super3} as we take the limit $b \to \infty$, by using \eqref{eq:super3a}, and by using $\lim_{b \to \infty} v(b, b) \ge 0$ from condition (i), we obtain
\begin{equation}\label{eq:super4}
v(x, m) \ge \pr^{x, m}(\tau_\alp < \infty),
\end{equation}
which implies $v \ge \psi$ on $\mD$.
\end{proof}

\subsection{Scaled Cram\'er-Lundberg model}\label{sec:scale_CL}

In this section, we scale the Cram\'er-Lundberg risk model by $n > 0$.  To obtain the scaled model, multiply the Poisson rate $\la$ by $n$, divide the claim severity by $\sqrtn$, and adjust the premium rate so that net premium income remains constant. Specifically, define $\lan = n \la$, so $n$ large is equivalent to $\lan$ large.  Scale the claim severity by defining $\Yn = Y/\sqrtn$. Also, define $\tetn = \tet/\sqrtn$ and $\cn = c + (\sqrtn - 1) \la \E Y$, which implies $\cn - \lan \En \Yn = c - \la \E Y$, independent of $n$.\footnote{By writing $\En$, we mean expectation with respect to the measure induced by $\Yn$.  Specifically,
\begin{align*}
\En(g(\Yn)) &= \int_0^\infty g(y) dF_{\Yn}(y) = \int_0^\infty g(y) d\Fy(\sqrtn y) \\
&= \int_0^\infty g(t/\sqrtn) d\Fy(t) = \E(g(Y/\sqrtn)),
\end{align*}
in which $\E = \E_1$.}  The parameter $\eta$ remains unchanged.  Finally, define
\begin{align*}
\kap_n &= (1 + \tetn) \lan \En\Yn + \dfrac{\eta}{2} \,  \lan \En(\Yn^2) - \cn \\
&= (1 + \tet) \la \E Y + \dfrac{\eta}{2} \,  \la \E(Y^2) - c = \kap,
\end{align*}
so $\kapn$ is also independent of $n$.

In the scaled model, define an {\it $n$-admissible retention strategy} $\mRn = \{ (\Rn)_t \}_{t \ge 0}$ is as in Definition \ref{def:adm}, except with condition (iii) replaced by
\begin{equation}\label{eq:iii_prime}
(iii') \hskip 20 pt   \cn - (1 + \tetn) \lan \En(\Yn - (\Rn)_t) - \dfrac{\eta}{2} \, \lan \En((\Yn - (\Rn)_t)^2) > \lan \En (\Rn)_t.
\end{equation}
Similarly, define an {\it $n$-admissible retention function} $\Rn$.  If we are given an $n$-admissible retention strategy $\mRn = \{ (\Rn)_t \}_{t \ge 0}$, then we can define an admissible retention strategy $\mR = \{ R_t \}_{t \ge 0}$ by
\begin{equation}\label{eq:R_Rn_t}
R_t(\omega, y) = \sqrtn \, (\Rn)_t(\omega, y/\sqrtn \, ).
\end{equation}
Indeed, by using the assignment in \eqref{eq:R_Rn_t}, condition (iii$'$) in \eqref{eq:iii_prime} implies
\begin{align*}
&(c + (\sqrtn - 1) \la \E Y) - \left(1 + \dfrac{\tet}{\sqrtn} \right) n \la \E\bigg( \dfrac{Y - R_t}{\sqrtn} \bigg)  - \dfrac{\eta}{2} \, n \la \E\ \bigg(\bigg( \dfrac{Y - R_t}{\sqrtn} \bigg)^2 \bigg) > n \la \E \bigg( \dfrac{R_t}{\sqrtn} \bigg) \\
&\iff c - \la(1 + \tet) \E Y + \la \tet \E R_t -  \dfrac{\eta}{2} \, \la \E((Y - R_t)^2) > 0 \\
&\iff c - \la(1 + \tet) \E (Y - R_t) -  \dfrac{\eta}{2} \, \la \E((Y - R_t)^2) > \la \E R_t.
\end{align*}
Conversely, given an admissible strategy $\mR$, we can define an $n$-admissible strategy $\mRn$ via \eqref{eq:R_Rn_t}.  Thus, we deduce there is a one-to-one correspondence between admissible strategies and $n$-admissible strategies.  Similarly, via the relationship
\begin{equation}\label{eq:R_Rn}
R(y) = \sqrtn \, \Rn(y/\sqrtn \, ),
\end{equation}
we obtain a one-to-one correspondence between admissible retention functions and $n$-admissible retention functions.

Let $\Xn^{\mRn}$ denote the surplus process for the scaled model; thus, $\Xn^{\mRn}$ follows the dynamics
\begin{align}
\label{eq:Xn}
d \big(\Xn^{\mRn} \big)_t &= \left( - \kapn + \lan \left( (1 + \tetn) \En (\Rn)_t + \eta \En( \Yn (\Rn)_t) - \dfrac{\eta}{2} \, \En((\Rn)^2_t) \right) \right)dt - (\Rn)_t \, d(\Nn)_t \notag \\
&= \left( - \kap + n \la \left( (1 + \tet/\sqrtn) \En (\Rn)_t + \eta \En(Y (\Rn)_t/\sqrtn) - \dfrac{\eta}{2} \, \En((\Rn)^2_t) \right) \right)dt - (\Rn)_t \, d(\Nn)_t,
\end{align}
in which $\Nn$ denotes the Poisson process with rate $\lan = n \la$.  Let $\psin$ denote the minimum probability of drawdown for the scaled system, and note that Theorems \ref{thm:sub} and \ref{thm:super} apply to $\psin$ with $\mL^R$ replaced by $\mLn^{\Rn}$, in which $\mLn^{\Rn}$ is defined as follows: for $u \in \mC^{1,1}(\mD)$ and for an $n$-admissible retention function $\Rn$,
\begin{align}
\label{eq:mLn}
&\mLn^{\Rn} u(x, m) =  - \kapn u_x(x, m) \notag \\
&+ \lan \left[ \left( (1 + \tetn) \En \Rn + \eta \En(\Yn\Rn) - \dfrac{\eta}{2} \, \En(\Rn^2) \right) u_x(x,m) + \En u(x - \Rn, m) - u(x, m) \right].
\end{align}
In \eqref{eq:mLn}, we again extend $u$ by defining $u(x, m) = 1$ for $x < \alp m$.

\subsection{Diffusion approximation}\label{sec:diff}

Let $\Xnd^\mR$ denote the diffusion approximation to the $n$-scaled process in \eqref{eq:Xn}, which we form by
\[
(\Rn)_t \, d(\Nn)_t \approx n \la \En (\Rn)_t \, dt - \sqrt{n \la \En((\Rn)^2_t)} \, dB_t,
\]
in which $B$ is a standard Brownian motion on the filtered probability space.  Then, $\Xnd^{\mRn}$ follows the dynamics
\begin{align}\label{eq:Xd}
d \big(\Xnd^{\mR_n} \big)_t &= \left( - \kap + n \la \left( (\tet/\sqrtn) \En(\Rn)_t + \eta \En(\Yn (\Rn)_t) - \dfrac{\eta}{2} \, \En ((\Rn)^2_t) \right) \right)dt + \sqrt{n \la \En((\Rn)^2_t)} \, dB_t \notag \\
&= \left( - \kap + \la \left( \tet \En (\sqrtn (\Rn)_t) + \eta \En(\sqrtn \Yn \cdot \sqrtn (\Rn)_t) - \dfrac{\eta}{2} \, \En(n (\Rn)^2_t) \right) \right)dt \notag \\
&\quad + \sqrt{\la \En(n(\Rn)^2_t)} \, dB_t \notag \\
&=  \left( - \kap + \la \left( \tet \E R_t + \eta \E(YR_t) - \dfrac{\eta}{2} \, \E(R^2_t) \right) \right)dt + \sqrt{\la \E(R^2_t)} \, dB_t,
\end{align}
in which we define the strategy $\mR = \{ R_t \}_{t \ge 0}$ by the assignment in \eqref{eq:R_Rn_t}.  Thus, the dynamics of $\Xnd^{\mRn}$ are independent of $n$ after this assignment, which implies that the minimum probability of ruin for the diffusion approximation of the $n$-scaled process is independent of $n$.

Let $\psid$ denote the minimum probability of drawdown for the diffusion approximation of the (scaled) Cram\'er-Lundberg model.  Han et al.\ \cite{HLY2020} solved the optimization problem associated with $\psid$.  The following theorem is a limiting case (as the riskless rate $r \to 0^+$) of Theorem 3.2, the main result of Han et al.\ \cite{HLY2020}.

\begin{theorem}\label{thm:psid}
The minimum probability of drawdown $\psid$ on
\[
\mD = \big\{(x,m) \in \R_+^2: \alp m \le x \le m \big\}
\]
under the diffusion approximation in \eqref{eq:Xd} equals
\begin{equation}\label{eq:psid}
\psid(x, m) = 1 - \hd(m) \left( 1 - e^{-\rhod(x - \alp m)} \right),
\end {equation}
in which $\hd$ is defined by
\begin{equation}\label{eq:hd}
\hd(m) = \left( 1 - e^{-\rhod(1 - \alp)m} \right)^{\frac{\alp}{1 -\alp}},
\end{equation}
and $\rhod > 0$ uniquely solves
\begin{equation}\label{eq:rhod}
c - \la \E Y = \la \rho  \int_0^\infty \left( \dfrac{\tet + \eta y}{\rho + \eta} \wedge y \right) \Sy(y) dy.
\end{equation}
The corresponding optimal retention strategy is a constant strategy $\{\Rd \}_{t \ge 0}$, in which
\begin{equation}\label{eq:Rd}
\Rd(y) = \dfrac{\tet + \eta y}{\rhod + \eta} \wedge y,
\end{equation}
with $y \ge 0$ the possible claim size.  \qed
\end{theorem}

\begin{remark}\label{rem:Rdn}
If we reverse \eqref{eq:R_Rn} as applied to $R = \Rd$ in \eqref{eq:Rd}, then we obtain
\begin{equation}\label{eq:Rdn}
\Rdn(y) := \dfrac{1}{\sqrtn} \Rd(\sqrtn y) = \dfrac{1}{\sqrtn} \left( \dfrac{\tet + \eta \sqrtn y}{\rhod + \eta} \wedge \sqrtn y \right) = \dfrac{\tetn + \eta y}{\rhod + \eta} \wedge y,
\end{equation}
that is, we get the optimal retention function for the diffusion approximation of the $n$-scaled model, as expected, from the discussion following \eqref{eq:Xd}.  \qed
\end{remark}

\begin{remark}
In Section {\rm \ref{sec:adj_coeff_diff}}, we show that $\rhod$ is the maximum adjustment coefficient, which parallels a similar result in Liang, Liang, and Young {\rm \cite{LLY2020}}.  \qed
\end{remark}

An important result of this paper is that $\psid$ and $\psin$ are approximately equal, specifically, to order $\mO(n^{-1/2})$.  The technique for showing that $\psid$ approximates $\psin$ is to modify $\psid$ by a function of order $\mO(n^{-1/2})$ and, then, prove that the modified function is a sub- or supersolution of $\psin$ via Theorem \ref{thm:sub} or \ref{thm:super}, respectively.  In the next section, we detail the steps in the remainder of the paper.

\subsection{Outline of paper}\label{sec:outline}

To make it easier to follow the material in Sections \ref{sec:adj_coeff} and \ref{sec:diff_approx}, we outline our steps in those sections:
\begin{enumerate}
\item{}  In Section \ref{sec:adj_coeff}, we define analogs of the so-called {\it adjustment coefficient} from risk theory; for background on the adjustment coefficient see, for example, Section 5.4 in Schmidli \cite{S2017}.
\begin{enumerate}
\item{} In Section \ref{sec:adj_coeff_CL}, we define an analog of the maximum adjustment coefficient for the scaled Cram\'er-Lundberg model, which we denote by $\rhon$.

\item{}  In Section \ref{sec:adj_coeff_diff}, we show that $\rhod$ from Section \ref{sec:diff} is the analog of the maximum adjustment coefficient for the diffusion approximation.

\item{}  In Section \ref{sec:relation}, we prove two lemmas that relate $\rhon$ and $\rhod$.  In particular, we show $\lim \limits_{n \to \infty} \rhon = \rhod$.
\end{enumerate}

\item{}  In Section \ref{sec:diff_approx}, we justify using the diffusion approximation for the classical risk process when analyzing the minimum probability of drawdown.
\begin{enumerate}
\item{}  In Section \ref{sec:bounds}, we modify $\psid$ by functions of order $\mO(n^{-1/2})$ to obtain upper and lower bounds of $\psin$; denote those bounds by $\un$ and $\elln$, respectively.  We use Theorems \ref{thm:sub} and \ref{thm:super} to prove that, indeed, $\elln \le \psin \le \un$ on $\R$.  Propositions \ref{prop:psin_lower} and \ref{prop:psin_upper} prove these two inequalities, respectively.

\item{}  In Section \ref{sec:psin_lim}, we use Propositions \ref{prop:psin_upper} and \ref{prop:psin_lower} to prove Theorem \ref{thm:psin_lim}, which states that, as $n$ goes to infinity, $\psin$ converges to $\psid$ uniformly on $\mD$ with rate of convergence of order $\mO(n^{-1/2})$.

\item{} Finally, in Section \ref{sec:Rd_opt}, we prove the main result of our paper, namely, that if the insurer uses the optimal strategy from the diffusion approximation, then the resulting probability of drawdown is $\mO(n^{-1/2})$-optimal.  This result thereby justifies using the diffusion approximation when minimizing the probability of drawdown.
\end{enumerate}
\end{enumerate}

\section{Analogs of the adjustment coefficient}\label{sec:adj_coeff}

\subsection{Scaled Cram\'er-Lundberg model}\label{sec:adj_coeff_CL}

In this section, we define the analog of the adjustment coefficient for the scaled Cram\'er-Lundberg model for the probability of drawdown.  We will use that analog to create an exponential upper bound of $\psin$.

For a given $n$-admissible retention function $\Rn$, formally obtain the adjustment coefficient $\rhon(\Rn) > 0$ by substituting $e^{-\rho x}$ for $u$ in $\mLn^{\Rn} u = 0$, including when $x - Y/\sqrtn < \alp m$.  When we perform this substitution, we obtain the following equation for $\rhon(\Rn)$:
\begin{equation*}
\left[- \kapn + \lan \left( (1 + \tetn) \En \Rn + \eta \En(\Yn \Rn) - \frac{\eta}{2} \, \En(\Rn^2) \right) \right] \rho - \lan \big( M_{\Rn}(\rho) - 1 \big) = 0,
\end{equation*}
in which $M_{\Rn}$ denotes the moment generating function of $\Rn$, which is finite in a neighborhood of $0$ because $\My$ is finite in a neighborhood of $0$.  After we substitute for the $n$-scaled parameters and set $R(y) = \sqrtn \Rn(y/\sqrtn)$ as in \eqref{eq:R_Rn}, this equation is equivalent to
\begin{equation}\label{eq:rhon_R}
n \la \gn(\rho; R) = \la \left(\tet \E R + \eta \E(YR) - \dfrac{\eta}{2} \, \E(R^2) \right) - \kap,
\end{equation}
in which we define $\gn$ by
\begin{equation}\label{eq:gn}
\gn(\rho; R) = \dfrac{1}{\rho} \, \left( \E e^{\frac{\rho}{\sqrtn} R} - 1 - \frac{\rho}{\sqrtn} \, \E R\right).
\end{equation}
Condition (iii) in Definition \ref{def:adm} for an admissible retention function implies that the right side of \eqref{eq:rhon_R} is positive; in fact, those two statements are equivalent.  Also, $\gn(\rho; R)$ increases from $0^+$ to infinity as $\rho$ increases from $0^+$ to infinity; thus, $\rhon(\Rn) > 0$ exists for any $n$-admissible retention function $\Rn$.

Let $\rhon$ denote the maximum adjustment coefficient for the classical risk model, in which we maximize over $n$-admissible retention functions $\Rn$.  By following the argument in Section 4.1 of Liang, Liang, and Young \cite{LLY2020}, we deduce that $\rhon$ solves the following maximization problem:
\begin{equation}\label{eq:rhon_HJB_0}
\sup_{\Rn} \left\{ \left[- \kapn + \lan \left( (1 + \tetn) \En \Rn + \eta \En(\Yn \Rn) - \frac{\eta}{2} \, \En(\Rn^2) \right) \right] \rho - \lan \big( M_{\Rn}(\rho) - 1 \big) \right\} = 0,
\end{equation}
in which we maximize over $n$-admissible retention functions $\Rn$.  By using the assignment in \eqref{eq:R_Rn}, this maximization problem is equivalent to
\begin{equation}\label{eq:rhon_HJB}
-\kap + \la \sup_{R} \left[  \left( \tet \E R + \eta \E(Y R) - \frac{\eta}{2} \, \E(R^2) \right) - n \gn(\rho; R) \right] = 0,
\end{equation}
in which we maximize over admissible retention functions $R$.

In the following proposition, we give an expression for $\rhon$ and the corresponding optimal $n$-retention function $\Rn^\rho$.

\begin{proposition}\label{prop:rhon}
The maximum adjustment coefficient $\rhon > 0$ for the $n$-scaled risk process in \eqref{eq:Xn} uniquely solves
\begin{equation}\label{eq:rhon}
c - \la \E Y = n \la \int_0^\infty \left( e^{\rho \Rn^\rho(y)} - 1 \right) S_{\Yn}(y) dy.
\end{equation}
in which $S_{\Yn}(y) = \Sy(\sqrtn y)$, for all $y \ge 0$,  and the corresponding optimal $n$-retention function $\Rn^\rho$ is given by
\begin{equation}\label{eq:Rn_rho}
\Rn^\rho(y) =
\begin{cases}
y, &\quad 0 \le y \le \dfrac{1}{\rhon} \ln(1 + \tetn), \\
R_c(y), &\quad y >  \dfrac{1}{\rhon} \ln(1 + \tetn).
\end{cases}
\end{equation}
In \eqref{eq:Rn_rho}, $R_c(y) \in [0, y)$ for $y > \frac{1}{\rhon} \ln(1 + \tetn)$ uniquely solves
\begin{equation}\label{eq:Rc}
(1 + \tetn) + \eta (y - R) =  e^{\rhon R}.
\end{equation}
\end{proposition}

\begin{proof}
Consider the $R$-dependent terms in \eqref{eq:rhon_HJB}; rewrite them as follows:
\[
\int_0^\infty \left[(\sqrtn + \tet) R(y) + \eta y R(y) - \dfrac{\eta}{2} \, R^2(y) - \dfrac{n}{\rho} \, e^{\frac{\rho}{\sqrtn} R(y)} \right] d\Fy(y).
\]
If we maximize the integrand $y$-by-$y$, subject to $0 \le R(y) \le y$, then the integral itself is maximized.  To that end, for a fixed value of $\rho$, define the function $j$ by
\begin{equation}\label{eq:j}
j(R) = (\sqrtn + \tet) R + \eta y R - \dfrac{\eta}{2} \, R^2 - \dfrac{n}{\rho} \, e^{\frac{\rho}{\sqrtn} R}.
\end{equation}
Then,
\[
j'(R) = (\sqrtn + \tet) + \eta (y - R) - \sqrtn \, e^{\frac{\rho}{\sqrtn} R},
\]
and
\[
j''(R) = - \eta - \rho e^{\frac{\rho}{\sqrtn} R} < 0.
\]
Thus, because $j$ is strictly concave with respect to $R$, the critical value (truncated on the left by $0$ and on the right by $y$) maximizes $j$.  As $R$ increases from $0$ to $y$, $j'(R)$ decreases from $\tet + \eta y \ge 0$ to
\begin{equation}\label{eq:j'(y)}
\sqrtn(1 - e^{\rho y/\sqrtn}) + \tet.
\end{equation}
If the expression in \eqref{eq:j'(y)} is negative, then the maximizing $R$ lies in $[0, y)$.  If it is non-negative, then the maximizing $R$ equals $y$.  Let $\check R(\cdot; \rho)$ denote the maximizing retention function for a given $\rho > 0$; then,
\begin{equation}\label{eq:Rrho}
\check R(y; \rho) =
\begin{cases}
y, &\quad 0 \le y \le \dfrac{\sqrtn}{\rho} \, \ln \bigg(1 + \dfrac{\tet}{\sqrtn} \bigg), \\
\check R_c(y; \rho), &\quad y > \dfrac{\sqrtn}{\rho} \, \ln \bigg(1 + \dfrac{\tet}{\sqrtn} \bigg),
\end{cases}
\end{equation}
in which $\check R_c(y; \rho) \in [0, y)$ for $y > \frac{\sqrtn}{\rho} \ln \big(1 + \frac{\tet}{\sqrtn} \big)$ uniquely solves
\begin{equation}\label{eq:checkRc}
(\sqrtn + \tet) + \eta (y - R) = \sqrtn \, e^{\frac{\rho}{\sqrtn} R}.
\end{equation}
As an aside, Liang, Liang, and Young \cite{LLY2020} show that $\lim_{n \to \infty} \check R(y; \rhon) = \Rd(y)$.

Next, substitute $\check R(\cdot; \rho)$ from \eqref{eq:Rrho} into \eqref{eq:rhon_HJB}, or equivalently,
\[
c - \la \E Y = \la \left( \tet \E(Y - R) + \dfrac{\eta}{2} \, \E((Y - R)^2) \right) + \dfrac{n \la}{\rho} \, \E \bigg( e^{\frac{\rho}{\sqrtn} R} - 1 - \dfrac{\rho R}{\sqrtn} \bigg),
\]
and solve for $\rhon$ to obtain
\begin{equation}\label{eq:rhon_sub}
c - \la \E Y = \sqrtn \la \int_0^\infty \left( \exp\bigg( \rho \, \dfrac{\check R(t; \rho)}{\sqrtn} \bigg) - 1 \right) \Sy(t) dt.
\end{equation}
By substituting $t = \sqrtn y$ in \eqref{eq:rhon_sub}, and by reversing the assignment in \eqref{eq:R_Rn}, we obtain $\Rn^\rho(y) = \frac{1}{\sqrtn} \check R(\sqrtn y; \rhon)$ in \eqref{eq:Rn_rho} with $\rhon$ solving \eqref{eq:rhon}.

It remains to show that \eqref{eq:rhon} has a unique positive solution $\rhon$.  To that end, consider the exponent in the integrand of \eqref{eq:rhon}, namely, $\rho \Rn(y; \rho)$, in which we define $\Rn(y; \rho)$ by replacing $\rhon$ in  \eqref{eq:Rn_rho} with a generic $\rho > 0$.  When $0 < y < \frac{1}{\rho} \ln(1 + \tetn)$,
\[
\dfrac{\partial}{\partial\rho} \big(\rho \Rn^\rho(y)\big) = \dfrac{\partial}{\partial \rho} (\rho y) = y > 0.
\]
When $y > \frac{1}{\rho} \ln(1 + \tetn)$,
\begin{align*}
\dfrac{\partial}{\partial\rho} \big(\rho \Rn^\rho(y)\big) &= \dfrac{\partial}{\partial \rho} \big(\rho R_c(y)\big) = R_c(y) + \rho \dfrac{\partial R_c(y)}{\partial \rho} \\
&= R_c(y) - \rho \, \dfrac{R_c(y) e^{\rho R_c(y)}}{\eta + \rho e^{\rho R_c(y)}} =  \dfrac{\eta R_c(y)}{\eta + \rho e^{\rho R_c(y)}} > 0.
\end{align*}
Thus, the right side of \eqref{eq:rhon} increases with respect to $\rho$.  As $\rho \to 0^+$, the right side of \eqref{eq:rhon} approaches $0$, which is less than the left side because we assume $c > \la \E Y$.  As $\rho \to \infty$, $\Rn(y; \rho)$ approaches $0$ in such a way that
\[
\lim_{\rho \to \infty} e^{\rho \Rn(y; \rho)} = (1 + \tetn) + \eta y.
\]
Thus, the right side of \eqref{eq:rhon} approaches
\begin{align*}
n \la \int_0^\infty \left( \tetn + \eta y \right) S_{\Yn}(y) dy = \la \int_0^\infty (\tet + \eta t) \Sy(t) dt = \la \left( \tet \E Y + \dfrac{\eta}{2} \, \E(Y^2) \right),
\end{align*}
which is greater than $c - \la \E Y$, from the assumption in \eqref{eq:c_not_huge}.  It follows that \eqref{eq:rhon} has a unique positive solution $\rhon$.
\end{proof}

\subsection{Diffusion approximation}\label{sec:adj_coeff_diff}

The analog of \eqref{eq:rhon_HJB} for the maximum adjustment coefficient for the diffusion approximation is
\begin{equation}\label{eq:rhod_HJB}
-\kap + \la \sup_{R} \left[ \tet \E R + \eta \E(Y R) - \frac{\rho + \eta}{2} \, \E(R^2) \right] = 0,
\end{equation}
in which we maximize over admissible retention functions $R$.  In the following proposition, we show that the maximum adjustment coefficient for the diffusion approximation equals $\rhod$ from Theorem \ref{thm:psid}, with $\Rd$ the optimal retention function.

\begin{proposition}\label{prop:rhod}
The maximum adjustment coefficient $\rhod > 0$ for the diffusion approximation in \eqref{eq:Xd} uniquely solves \eqref{eq:rhod},  and the corresponding optimal retention function equals $\Rd$ in \eqref{eq:Rd}.
\end{proposition}

\begin{proof}
By following the argument in the proof of Proposition \ref{prop:rhon}, for a given value of $\rho > 0$, the retention function maximizing the expression in \eqref{eq:rhod_HJB} equals
\[
R(y; \rho) = \dfrac{\tet + \eta y}{\rho + \eta} \wedge y.
\]
By substituting this expression into \eqref{eq:rhod_HJB}, we obtain equation \eqref{eq:rhod} for $\rhod$.  The remainder of the proof is similar to, but simpler than, the proof of Proposition \ref{prop:rhon}, so we omit the details.
\end{proof}

As we expect from the maximization problems in \eqref{eq:rhon_HJB} and \eqref{eq:rhod_HJB}, the maximum adjustment coefficients $\rhon$ and $\rhod$ are related, which we prove in the next section.

\subsection{Relationship between $\rhon$ and $\rhod$}\label{sec:relation}

In this section, we prove two lemmas that relate $\rhon$ and $\rhod$.  We use those lemmas in Section \ref{sec:bounds} to modify $\psid$ by function of order $\mO(n^{-1/2})$ to bound $\psin$.

In the first lemma, we prove $\rhon < \rhod$.

\begin{lemma}\label{lem:rhon_le_rhod}
The maximum adjustment coefficient for the $n$-scaled risk process is less than the maximum adjustment coefficient for the diffusion approximation, that is, $\rhon < \rhod$.
\end{lemma}

\begin{proof}
For a given admissible retention function $R$, let $\rhon(R)$ and $\rhod(R)$ denote the solutions of \eqref{eq:rhon_HJB} and \eqref{eq:rhod_HJB}, respectively, with the $\sup_R$ removed.  We begin by showing that $\rhon(R) < \rhod(R)$.  (As an aside, the left side of both \eqref{eq:rhon_HJB} and \eqref{eq:rhod_HJB} decrease from a positive number to negative infinity as $\rho$ increases from $0$ to infinity.  Therefore, positive solutions $\rhon(R)$ and $\rhod(R)$ exist.)

$\rhod(R) > 0$ solves
\begin{equation}\label{eq:rhod_R}
-\kap + \la \left[ \tet \E R + \eta \E(Y R) - \frac{\rho + \eta}{2} \, \E(R^2) \right] = 0.
\end{equation}
Because the left side of this equation decreases with respect to $\rho$, then $\rhon(R) < \rhod(R)$ if and only if
\[
-\kap + \la \left[ \tet \E R + \eta \E(Y R) - \frac{\rhon(R) + \eta}{2} \, \E(R^2) \right] > 0.
\]
By using \eqref{eq:rhon_R} and canceling a factor of $\la > 0$, this inequality becomes
\[
n \, \gn(\rhon(R); R) - \dfrac{\rhon(R)}{2} \, \E(R^2) > 0,
\]
or equivalently,
\[
\dfrac{n}{\rhon(R)} \, \E \bigg( e^{\frac{\rhon(R) \cdot R}{\sqrtn}}  - 1 - \frac{\rhon(R) \cdot R}{\sqrtn} - \frac{\rhon^2(R) \cdot R^2}{n} \bigg) > 0,
\]
which is true because $e^x > 1 + x + x^2$ for $x > 0$.

Thus, we have shown $\rhon(R) < \rhod(R)$ for all admissible retention function $R$.  Now, let $R = \check R$ from \eqref{eq:Rrho}; then, we have
\[
\rhon = \rhon(\check R) < \rhod(\check R) \le \rhod,
\]
and we have proved this lemma.
\end{proof}

In the following lemma, we show that we can modify $\rhod$ by a constant of order $\mO(n^{-1/2})$ to get a lower bound of $\rhon$.

\begin{lemma}\label{lem:rhod_le_rhon}
Choose $C$ so that
\begin{equation}\label{eq:C}
C > \dfrac{\E(\Rd^3)}{3 \, \E(\Rd^2)} \, \rhod^2.
\end{equation}
Then, there exists $N > 0$ such that, for all $n \ge N$,
\begin{equation}\label{eq:rhod_le_rhon}
0 < \rhod - \dfrac{C}{\sqrtn} < \rhon < \rhod,
\end{equation}
from which it follows that
\begin{equation}\label{eq:rhon_lim}
\lim \limits_{n \to \infty} \rhon = \rhod,
\end{equation}
with rate of convergence of order $\mO(n^{-1/2})$.
\end{lemma}

\begin{proof}
The limit in \eqref{eq:rhon_lim} follows directly from the bounds in \eqref{eq:rhod_le_rhon}.  The third inequality in \eqref{eq:rhod_le_rhon} follows from Lemma \ref{lem:rhon_le_rhod}; thus, it remains to show the first and second inequalities in \eqref{eq:rhod_le_rhon}.  To that end, consider the admissible retention strategy $\Rd$ in \eqref{eq:Rd}.  Then, from \eqref{eq:rhon_R} and \eqref{eq:gn}, we know $\rhon(\Rd)$ solves
\begin{equation}\label{eq:lem32_0}
- \kap + \la \left[ \tet \E \Rd + \eta \E(Y \Rd) - \dfrac{\eta}{2} \, \E(\Rd^2) - \dfrac{n}{\rho} \left( \E e^{\frac{\rho}{\sqrtn} \Rd} - 1 - \frac{\rho}{\sqrtn} \, \E \Rd \right) \right] = 0,
\end{equation}
or equivalently, from the expression for $\rhod$ in \eqref{eq:rhod_HJB}, $\rhon(\Rd)$ solves
\begin{equation}\label{eq:lem32_1}
\dfrac{\E(\Rd^2)}{2} \, \rhod - n \gn(\rho; \Rd) = 0 
\end{equation}
The left side of this equation decreases with respect to $\rho$; thus, to show that
\[
\rhod - \dfrac{C}{\sqrtn} < \rhon(\Rd)
\]
for some value of $C$, it is enough to show that the left side of \eqref{eq:lem32_1} is positive when we set $\rho = \rhod - \frac{C}{\sqrtn}$.  That is, we want to show there exist $C$ and $N$ such that
\[
\dfrac{\E(\Rd^2)}{2} \, \rhod > n \gn(\rhod - C/\sqrtn; \Rd),
\]
for all $n \ge N$, with $\rhod - \frac{C}{\sqrtn} > 0$, or equivalently,
\begin{equation}\label{eq:lem32_2}
\dfrac{\E(\Rd^2)}{2} \, \rhod > \dfrac{n}{\rhod - C/\sqrtn} \left( \E e^{\frac{\rhod - C/\sqrtn}{\sqrtn} \Rd} - 1 - \frac{\rhod - C/\sqrtn}{\sqrtn} \, \E \Rd \right).
\end{equation}

To simplify the right side of inequality \eqref{eq:lem32_2}, we use the following identity:
\[
e^x - 1 - x = \dfrac{x^2}{2} + \dfrac{x^3}{6} + \dfrac{x^4}{6} \int_0^1 (1 - \ome)^3 e^{\ome x} d\ome.
\]
Then, inequality \eqref{eq:lem32_2} is equivalent to
\begin{align*}
&\dfrac{\E(\Rd^2)}{2} \, \rhod \left( \rhod - \frac{C}{\sqrtn} \right) \\
&> \E \bigg( \dfrac{(\rhod - C/\sqrtn)^2 \Rd^2}{2} + \dfrac{(\rhod - C/\sqrtn)^3 \Rd^3}{6\sqrtn} + \dfrac{(\rhod - C/\sqrtn)^4 \Rd^4}{6n} \int_0^1 (1 - \ome)^3 e^{\ome \frac{\rhod - C/\sqrtn}{\sqrtn} \Rd} d \ome \bigg),
\end{align*}
or
\[
\E(\Rd^2) \, \dfrac{C}{\sqrtn} > \E \bigg( \dfrac{(\rhod - C/\sqrtn)^2 \Rd^3}{3 \sqrtn} + \dfrac{(\rhod - C/\sqrtn)^3 \Rd^4}{3n} \int_0^1 (1 - \ome)^3 e^{\ome \frac{\rhod - C/\sqrtn}{\sqrtn} \Rd} d \ome \bigg),
\]
which holds if the following stronger inequality holds:
\[
\E(\Rd^2) \, C > \E \bigg( \dfrac{\rhod^2 \Rd^3}{3} + \dfrac{\rhod^3 \Rd^4}{3 \sqrtn} \, e^{\frac{\rhod}{\sqrtn} \Rd} \bigg),
\]
or
\begin{equation}\label{eq:lem32_3}
\E(\Rd^2) \, C - \rhod^2 \, \dfrac{\E(\Rd^3)}{3} > \dfrac{\rhod^3}{3 \sqrtn} \, \E \Big(\Rd^4 \, e^{\frac{\rhod}{\sqrtn} \Rd} \Big).
\end{equation}
If we choose $C$ according to \eqref{eq:C}, then there exists $N > 0$ such that $\rhod - C/\sqrt{N} > 0$ and such that inequality \eqref{eq:lem32_3} holds at $n = N$.  Because the right side of \eqref{eq:lem32_3} decreases with $n$, it follows that inequality \eqref{eq:lem32_3} holds for all $n \ge N$.

We have, thereby, shown that
\begin{equation}\label{eq:rhod_le_rhon_Rd}
0 < \rhod - \dfrac{C}{\sqrtn} < \rhon(\Rd),
\end{equation}
for all $n \ge N$, and we know that $\rhon(\Rd) \le \rhon$ because $\rhon$ is maximal.  Thus, we have proven the first two inequalities in \eqref{eq:rhod_le_rhon}.
\end{proof}

\section{Justifying the diffusion approximation}\label{sec:diff_approx}

\subsection{Bounds for $\psin$}\label{sec:bounds}

In this section, we modify $\psid$ by functions of order $\mO(n^{-1/2})$ to obtain upper and lower bounds of $\psin$.  In the process of finding an upper bound of $\psin$, we obtain a type of Lundberg bound for $\psin$.  To that end, by analogy with the expression for $\psid$ in \eqref{eq:psid}, define $\opsin$ for $x \le m$ and $m \ge 0$ as follows:
\begin{equation}\label{eq:opsin}
\opsin(x, m) =
\begin{cases}
1 - \hn(m) \left( 1 - e^{-\rhon(x - \alp m)} \right), &\quad (x, m) \in \mD,\\
1, &\quad x < \alp m,
\end{cases}
\end{equation}
in which $\hn$ is defined by
\begin{equation}\label{eq:hn}
\hn(m) = \left( 1 - e^{-\rhon(1 - \alp)m} \right)^{\frac{\alp}{1 -\alp}},
\end{equation}
for all $m \ge 0$.  In the following lemma, we use Theorem \ref{thm:super} to prove that $\opsin$ is an upper bound of $\psin$.

\begin{lemma}\label{lem:opsin_upper}
For all $(x, m) \in \mD$,
\begin{equation}\label{eq:opsin_upper}
\psin(x, m) \le \opsin(x, m),
\end{equation}
in which $\psin$ is the minimum probability of drawdown for the $n$-scaled model, and $\opsin$ is defined in \eqref{eq:opsin}.
\end{lemma}

\begin{proof}
We prove this lemma via Theorem \ref{thm:super} modified to account for the $n$-scaled model, which essentially replaces $\mL^{\hR}$ in condition (iv) with $\mLn^{\hR_n}$, given in \eqref{eq:mLn}.  First, note that $\opsin \in \mC^{1,1}(\mD)$ by its definition.  Next, we go through each condition in Theorem \ref{thm:super} in turn.

\medskip

\noindent {\it Condition} (i):
\begin{align*}
\lim_{m \to \infty} \opsin(m, m) &= \lim_{m \to \infty} \left( 1 - \hn(m) \left( 1 - e^{-\rhon(1 - \alp)m} \right) \right) \\
&= 1 - \lim_{m \to \infty} \left( 1 - e^{-\rhon(1 - \alp)m} \right)^{\frac{\alp}{1 -\alp} + 1} = 1 - 1 = 0,
\end{align*}
so condition (i) is satisfied with equality.

\medskip

\noindent {\it Condition} (ii): This condition, namely, that $\opsin(x, m)$ is defined for all $x \le m$ and $m \ge 0$, with $\opsin(x, m) = 1$ for all $x < \alp m$, is satisfied by the definition of $\opsin$ in \eqref{eq:opsin}.

\medskip

\noindent {\it Condition} (iii):  For $m > 0$, differentiate $\opsin$ with respect to $m$ and simplify the expression to obtain
\begin{align*}
(\opsin)_m(x, m) = \dfrac{\alp \rhon \hn(m)}{1 - e^{-\rhon(1 - \alp)m}} \left(e^{-\rhon(x - \alp m)} - e^{-\rhon(1 - \alp)m}  \right),
\end{align*}
which equals $0$ when $x = m$.  Therefore, condition (iii) is satisfied with equality.

\medskip

\noindent {\it Condition} (iv):  Let $\hR_n$ be the $n$-admissible retention function $\Rn^\rho$ that maximizes the adjustment coefficient.  Then, by \eqref{eq:rhon_HJB_0}, $\Rn^\rho$ and $\rhon$ satisfy
\begin{align*}
\lan \left( \En e^{\rhon \Rn^\rho} - 1 \right) = \left[- \kapn + \lan \left( (1 + \tetn) \En \Rn^\rho + \eta \En(\Yn \Rn^\rho) - \frac{\eta}{2} \, \En((\Rn^\rho)^2) \right) \right] \rhon.
\end{align*}
Then, for $(x, m) \in \mD$, the expression in \eqref{eq:mLn} gives us
\begin{align*}
\mLn^{\Rn^\rho} \opsin(x, m) &= \left[- \kapn + \lan \left( (1 + \tetn) \En \Rn^\rho + \eta \En(\Yn \Rn^\rho) - \frac{\eta}{2} \, \En((\Rn^\rho)^2) \right) \right]  (\opsin)_x(x, m) \\
&\quad + \lan\big(  \En \opsin(x - \Rn^\rho, m) - \opsin(x, m) \big) \\
&= \left[- \kapn + \lan \left( (1 + \tetn) \En \Rn^\rho + \eta \En(\Yn \Rn^\rho) - \frac{\eta}{2} \, \En((\Rn^\rho)^2) \right) \right] \big(- \rhon \hn(m) e^{-\rhon(x - \alp m)} \big) \\
&\quad + \lan \left[ \En \Big( \left(1 - \hn(m) \left( 1 - e^{-\rhon(x - \Rn^\rho - \alp m)} \right)\right) \mathds{1}_{\{ x - \Rn^\rho \ge \alp m\}} \Big) \right] \\
&\quad + \lan \left[ \En \big( \mathds{1}_{\{ x - \Rn^\rho < \alp m\}} \big) - \left(1 - \hn(m) \left( 1 - e^{-\rhon(x - \alp m)} \right) \right) \right] \\
&= - \lan \hn(m) e^{-\rhon(x - \alp m)} \left( \En e^{\rhon \Rn^\rho} - 1 \right) + \lan \hn(m) \left( 1 - \En\big( \mathds{1}_{\{ x - \Rn^\rho \ge \alp m\}} \big) \right) \\
&\quad + \lan \hn(m) e^{-\rhon(x - \alp m)} \left( \En \big( e^{\rhon \Rn^\rho} \mathds{1}_{\{ x - \Rn^\rho \ge \alp m\}} \big) - 1 \right) \\
&= \lan \hn(m) \left\{ - \En \big( e^{-\rhon(x - \Rn^\rho - \alp m)} \mathds{1}_{\{ x - \Rn^\rho < \alp m\}} \big) + \En\big( \mathds{1}_{\{ x - \Rn^\rho < \alp m\}} \big) \right\} \\
&= - \lan \hn(m) \En \Big( \big( e^{\rhon(\alp m - (x - \Rn^\rho))} - 1 \big) \mathds{1}_{\{ x - \Rn^\rho < \alp m\}} \Big) \\
&\le 0.
\end{align*}
Thus, we have proved condition (iv).  Theorem \ref{thm:super}, then, implies inequality \eqref{eq:opsin_upper} on $\mD$.
\end{proof}

We obtain the following proposition from Lemmas \ref{lem:rhod_le_rhon} and \ref{lem:opsin_upper}, in which we modify $\psid$ in \eqref{eq:psid} by a function of order $\mO(n^{-1/2})$ to obtain an upper bound of $\psin$.

\begin{proposition}\label{prop:psin_upper}
Let $C$ and $N$ be as in the statement of Lemma {\rm \ref{lem:rhod_le_rhon}}, and for $n \in \N$, define $\un$ by
\begin{equation}\label{eq:un}
\un(x, m) =
\begin{cases}
1 - \kn(m) \left( 1 - e^{-(\rhod - C/\sqrtn)(x - \alp m)} \right), &\quad (x, m) \in \mD,\\
1, &\quad x < \alp m,
\end{cases}
\end{equation}
in which $\kn$ is defined by
\begin{equation}\label{eq:kn}
\kn(m) = \left( 1 - e^{-(\rhod - C/\sqrtn)(1 - \alp)m} \right)^{\frac{\alp}{1 -\alp}},
\end{equation}
Then, for $n \ge N$,
\begin{equation}\label{eq:psin_upper}
\psin \le \un,
\end{equation}
on $\mD$.
\end{proposition}

\begin{proof}
It is straightforward to show that $\opsin$ in \eqref{eq:opsin} decreases with respect to $\rhon > 0$ on $\mD$.  Thus, if we replace $\rhon$ in $\opsin$'s definition with a (positive) parameter less than $\rhon$, then we get a function that is an upper bound of $\opsin$.  That is exactly how we defined $\un$ in \eqref{eq:un} because, from Lemma \ref{lem:rhod_le_rhon}, we know $0 < \rhod - C/\sqrtn < \rhon$ for all $n \ge N$; thus, we have, from Lemma \ref{lem:opsin_upper},
\[
\psin \le \opsin \le \un,
\]
on $\mD$.
\end{proof}

In the following proposition, we modify $\psid$ to obtain a lower bound of $\psin$.

\begin{proposition}\label{prop:psin_lower}
Formally, define the random variable $\Zd = (Y - d) \big| (Y > d)$ for $d \ge 0$, and suppose $\varsigma$ exists such that $\My \big(\rhod/\sqrt{\varsigma} \, \big) < \infty$, with
\begin{equation}\label{eq:Zd}
\sup \limits_{d \ge 0} \E \Big( e^{\frac{\rhod}{\sqrt{\varsigma}} \Zd} \Big) < \infty.
\end{equation}
Choose $\eps > 0$, and define $\del$ by
\begin{equation}\label{eq:del}
\del = \sup \limits_{d \ge 0} \big( \rhod \E \Zd + \eps \big),
\end{equation}
and choose $N > \max \big( \del^2, 4 \varsigma \big)$ such that\footnote{Condition \eqref{eq:Zd} implies that we can find such an $N$.  Indeed, for $z$ large enough, we have
\[
z^2 < \dfrac{e^{\frac{\rhod z}{\sqrt{\varsigma}}}}{e^{\frac{\rhod z}{2\sqrt{\varsigma}}}} = e^{\frac{\rhod z}{2\sqrt{\varsigma}}},
\]
which implies there exists $M > 0$ such that
\[
\E\Big( \Zd^2 \, e^{\frac{\rhod}{2\sqrt{\varsigma}} \Zd} \Big) \le \E\Big( e^{\frac{\rhod}{\sqrt{\varsigma}} \Zd}\mathds{1}_{\{\Zd>M\}}  \Big)
+ \E\Big( M^2 e^{\frac{\rhod}{\sqrt{\varsigma}} M}\mathds{1}_{\{\Zd \le M\}}  \Big).
\]}
\begin{equation}\label{eq:N}
\sup \limits_{d \ge 0} \dfrac{\rhod^2}{\sqrt{N}} \, \E\Big( \Zd^2 \, e^{\frac{\rhod}{\sqrt{N}} \Zd} \Big) \le \eps.
\end{equation}
For $n \in \N$, define $\elln$ by
\begin{equation}\label{eq:elln}
\elln(x, m)=
\begin{cases}
\left(1 - \dfrac{\del}{\sqrtn} \right) \psid(x, m), &\quad (x, m) \in \mD,\\
1, &\quad x < \alp m.
\end{cases}
\end{equation}
Then, for all $n \ge N$,
\begin{equation}\label{eq:psio_scale}
\elln \le \psin,
\end{equation}
on $\mD$.
\end{proposition}

\begin{proof}
We prove this lemma via Theorem \ref{thm:sub} modified to account for the $n$-scaled model, which essentially replaces $\mL^{\hR}$ in condition (iv) with $\mLn^{\hR_n}$, given in \eqref{eq:mLn}.  First, note that $\elln \in \mC^{1,1}(\mD)$ by its definition.  Next, we go through each condition in Theorem \ref{thm:sub} in turn.

Without loss of generality, assume $n > \del^2$.

\medskip

\noindent {\it Condition} (i):
\begin{align*}
\lim_{m \to \infty} \elln(m, m) &= \left(1 - \dfrac{\del}{\sqrtn} \right) \lim_{m \to \infty} \psid(m, m) = 0,
\end{align*}
so condition (i) is satisfied with equality.

\medskip

\noindent {\it Condition} (ii):  This condition, namely, that $\elln(x, m)$ is defined for all $x \le m$ and $m \ge 0$, with $\elln(x, m) = 1$ for all $x < \alp m$, is satisfied by the definition of $\elln$ in \eqref{eq:elln}.

\medskip

\noindent {\it Condition} (iii):  For $m > 0$, differentiate $\elln$ with respect to $m$ and simplify the expression to obtain
\begin{align*}
(\elln)_m(x, m) = \left(1 - \dfrac{\del}{\sqrtn} \right) \dfrac{\alp \rhod \hd(m)}{1 - e^{-\rhod(1 - \alp)m}} \left(e^{-\rhod(x - \alp m)} - e^{-\rhod(1 - \alp)m}  \right),
\end{align*}
which equals $0$ when $x = m$.  Therefore, condition (iii) is satisfied with equality.

\medskip

\noindent {\it Condition} (iv):  Let $\Rn$ be any $n$-admissible retention function. We wish to show that $\mLn^{\Rn} \elln(x, m) \ge 0$ for all $(x, m) \in \mD$.
\begin{align}
\mLn^{\Rn} \elln(x, m) &= \left[- \kapn + \lan \left( (1 + \tetn) \En \Rn + \eta \En(\Yn \Rn) - \frac{\eta}{2} \, \En(\Rn^2) \right) \right]  (\elln)_x(x, m) \nonumber \\
&\quad + \lan\big(  \En \elln(x - \Rn, m) - \elln(x, m) \big). \label{eq:operLn}
\end{align}
We compute
\begin{align}\label{eq:lnx}
(\elln)_x(x, m) = - \left( 1 - \dfrac{\del}{\sqrtn} \right) \rhod \hd(m) e^{-\rhod(x - \alp m)},
\end{align}
and
\begin{align}\label{eq:dfln}
\En \elln(x - \Rn, m) - \elln(x, m) &= \En \bigg( \bigg( 1 - \dfrac{\del}{\sqrtn} \bigg) \Big( 1 - h_D(m)\Big(1 - e^{- \rho_D (x - \Rn - \alp m)}\Big) \Big) \mathds{1}_{\{x-\Rn \ge \alp m\}} \nonumber \\
& \qquad \quad + \mathds{1}_{\{ x-\Rn < \alp m\}} - \bigg( 1 - \dfrac{\del}{\sqrtn} \bigg)\Big( 1 - h_D(m)\Big(1 - e^{- \rho_D (x - \alp m)}\Big) \Big)\bigg).
\end{align}
Now, by substituting \eqref{eq:lnx} and \eqref{eq:dfln} into \eqref{eq:operLn} and by rearranging terms, we obtain
\begin{align*}
\mLn^{\Rn} \elln(x, m) & \propto \kap \rho_D - \la \left( (\sqrtn + \tet) \E R + \eta \E(Y R) - \frac{\eta}{2} \, \E(R^2) \right) \rho_D \\
&\quad + \frac{n \la e^{\rhod (x - \alp m)}}{h_D(m)} \, \E\Big(\Big(1 - h_D(m)\Big(1 - e^{- \rho_D (x - R/\sqrtn - \alp m)}\Big) \Big) \mathds{1}_{\{x - R/\sqrtn \ge \alp m \}} \Big) \\
&\quad + \frac{n \la e^{\rhod (x - \alp m)}}{h_D(m)} \, \dfrac{\E \big(\mathds{1}_{\{ x - R/\sqrtn < \alp m\}}\big)}{1 - {\del}/{\sqrtn} }
 - \frac{n \la e^{\rhod (x - \alp m)}}{h_D(m)}  \E\left(1 - h_D(m)\left(1 - e^{- \rho_D (x - \alp m)}\right) \right),
\end{align*}
in which $\propto$ denotes ``{\it positively} proportional to,'' and in which $R$ and $\Rn$ are related via \eqref{eq:R_Rn}.
Recall that $\rhod$ solves
\[
 \kap \rhod = \la \sup \limits_R \left\{ \left( \tet \E R + \eta \E \big( YR \big) - \dfrac{\eta}{2} \, \E \big( R^2 \big) \right) \rhod - \dfrac{1}{2} \, \E \big(R^2 \big) \rhod^2 \right\},
\]
that is,
\[
 \kap \rhod - \la  \left\{ \left( \tet \E R + \eta \E \big( YR \big) - \dfrac{\eta}{2} \, \E \big( R^2 \big) \right) \rhod - \dfrac{1}{2} \, \E \big(R^2 \big) \rhod^2 \right\} \ge 0,
\]
for any admissible retention function $R$.  Thus, to prove $\mLn^{\Rn} \elln(x, m) \ge 0$ for all $(x, m) \in \mD$ and for all $n$-admissible retention functions $\Rn$, it is enough to prove
\begin{align}\label{eq:Fn_elln_pos2}
& \la \left\{ \left( \tet \E R + \eta \E \big( Y R \big) - \dfrac{\eta}{2} \, \E \big( R^2 \big) \right) \rhod - \dfrac{1}{2} \, \E \big( R^2 \big) \rhod^2 \right\} \\
&\ge \la \left [\left( (\sqrtn + \tet) \E R + \eta \E(Y R) - \frac{\eta}{2} \, \E(R^2) \right) \right] \rho_D \notag \\
& \quad - \frac{n \la e^{\rhod (x - \alp m)}}{h_D(m)}  \E\Big( \Big(1 - h_D(m)\Big(1 - e^{- \rho_D (x - R/\sqrtn - \alp m)}\Big) \Big) \mathds{1}_{\{x - R/\sqrtn \ge \alp m\}}\Big) \notag\\
& \quad - \frac{n \la e^{\rhod (x - \alp m)}}{h_D(m)}  \dfrac{\E \big(\mathds{1}_{\{ x - R/\sqrtn < \alp m\}}\big)}{1 - {\del}/{\sqrtn} }
 + \frac{n \la e^{\rhod (x - \alp m)}}{h_D(m)}  \E\left(1 - h_D(m)\left(1 - e^{- \rho_D (x - \alp m)}\right) \right), \notag
\end{align}
for all admissible retention functions $R$, or equivalently,
\begin{align}\label{eq:Fn_elln_pos3}
&\dfrac{1}{2n} \, \E \big(R^2 \big) \rhod^2  \notag \\
&\le - 1 - \dfrac{\rhod \E R}{\sqrtn} + e^{\rhod (x - \alp m)} \left(\frac{\del /\sqrtn}{h_D(m)(1 - \del/\sqrtn)} + 1 \right)\E \big(\mathds{1}_{\{ x - R/\sqrtn < \alp m\}} \big) + \E \Big(e^{ \frac{\rhod R}{\sqrtn}}\mathds{1}_{\{ x - R/\sqrtn \ge \alp m\}}\Big) \notag\\
&=  \int_0^\infty  \left( e^{\frac{\rhod R(y)}{\sqrtn}} - 1 - \dfrac{\rhod R(y)}{\sqrtn} \right) d\Fy(y)  \notag \\
&\quad +  \int_{0}^\infty \left( e^{\rhod (x-\alp m)}\left(\frac{\del /\sqrtn}{h_D(m)(1 - \del/\sqrtn)} + 1 \right) - e^{\frac{\rhod R(y)}{\sqrtn}} \right)\mathds{1}_{\{ x - R/\sqrtn < \alp m\}} d\Fy(y).
\end{align}
From $e^x > 1 - x - x^2/2$ for all $x > 0$, we deduce
\[
\dfrac{1}{2n} \, \E \big(R^2 \big) \rhod^2 <  \int_0^\infty  \left( e^{\frac{\rhod R(y)}{\sqrtn}} - 1 - \dfrac{\rhod R(y)}{\sqrtn} \right) d\Fy(y).
\]
Thus, to prove \eqref{eq:Fn_elln_pos3}, it is enough to prove the stronger inequality
\begin{align}\label{ieq:whR}
\int_{0}^\infty \left(e^{\frac{\rhod}{\sqrtn}\left(R(y) - \sqrtn (x - \alp m)\right)} - \frac{\del /\sqrtn}{h_D(m)(1 - \del/\sqrtn)} - 1 \right) \mathds{1}_{\{ x - R(y)/\sqrtn < \alp m\}} d\Fy(y) \le 0,
\end{align}
and we wish to find values of $\del$ and $N > \del^2$ for which inequality \eqref{ieq:whR} holds for all $n > N$ and $x > 0$.  Note that the maximal retention function of the integrand in \eqref{ieq:whR} is $R(y) = y$; thus, to prove \eqref{ieq:whR}, it is sufficient to prove
\begin{align*}
\int_{0}^\infty \left(e^{\frac{\rhod}{\sqrtn}\left(y - \sqrtn (x - \alp m)\right)} - \frac{\del /\sqrtn}{h_D(m)(1 - \del/\sqrtn)} - 1 \right) \mathds{1}_{\{ x - y/\sqrtn < \alp m\}} d\Fy(y) \le 0,
\end{align*}
that is,
\begin{align*}
\int_{\sqrtn (x - \alp m)}^\infty \left(e^{\frac{\rhod}{\sqrtn}\left(y - \sqrtn (x - \alp m)\right)} - \frac{\del /\sqrtn}{h_D(m)(1 - \del/\sqrtn)} - 1 \right) d\Fy(y) \le 0.
\end{align*}
Let $d$ denote $\sqrtn (x - \alp m)$.  If $\Sy(d) = 0$, then the left side is identically $0$, so suppose $\Sy(d) > 0$.  After replacing $\sqrtn (x - \alp m)$ by $d$ and dividing by $\Sy(d)$, the above inequality becomes
\[
\int_{d}^\infty \left(e^{\frac{\rhod}{\sqrtn}\left(y - d\right)} - \frac{\del /\sqrtn}{h_D(m)(1 - \del/\sqrtn)} - 1 \right) \dfrac{d\Fy(y)}{\Sy(d)} \le 0,
\]
for $d \ge 0$, or equivalently,
\begin{align}\label{ineq:ySF}
\int_d^\infty \left( e^{ \frac{\rhod}{\sqrtn}(y - d)} - 1 \right) \dfrac{d\Fy(y)}{\Sy(d)} \le \frac{\del /\sqrtn}{h_D(m)(1 - \del/\sqrtn)} \, .
\end{align}
Formally, define $\Zd = (Y - d) \big| (Y > d)$; then, inequality \eqref{ineq:ySF} becomes
\[
\int_0^\infty \left(  e^{ \frac{\rhod z}{\sqrtn}} - 1 \right) dF_{\Zd}(z) \le \frac{\del /\sqrtn}{h_D(m)(1 - \del/\sqrtn)} \, .
\]
Note that $0<h_D(m)<1$; thus, if we find $\del$ to satisfy the following even stronger inequality, then the above sequence of inequalities holds:
\begin{equation}\label{ineq:del1}
\int_0^\infty \left(  e^{ \frac{\rhod z}{\sqrtn}} - 1 \right) dF_{\Zd}(z) \le \dfrac{\del}{\sqrtn} \, .
\end{equation}
Rewrite the integrand from the left side of inequality \eqref{ineq:del1} as follows:
\[
e^{ \frac{\rhod z}{\sqrtn}} - 1 = \dfrac{\rhod z}{\sqrtn} + \dfrac{\rhod^2 z^2}{n} \int_0^1 (1 - \ome) e^{\frac{\rhod z}{\sqrtn} \, \ome} d\ome.
\]
Thus, inequality \eqref{ineq:del1} is equivalent to
\[
\int_0^\infty \left(  \dfrac{\rhod z}{\sqrtn} + \dfrac{\rhod^2 z^2}{n} \int_0^1 (1 - \ome) e^{\frac{\rhod z}{\sqrtn} \, \ome} d \ome \right) dF_{\Zd}(z) \le \dfrac{\del}{\sqrtn} \, ,
\]
or, after multiplying both side by $\sqrtn$ and switching the order of integration,
\begin{equation*}
\rhod \E \Zd + \dfrac{\rhod^2}{\sqrtn} \int_0^1 (1 - \ome) \, \E \Big( \Zd^2 \, e^{\frac{\rhod \ome}{\sqrt{n}} \Zd} \Big) d \ome \le \del,
\end{equation*}
or more strongly,
\begin{equation}\label{ineq:del2}
\rhod \E \Zd + \dfrac{\rhod^2}{\sqrtn} \, \E \Big( \Zd^2 \, e^{\frac{\rhod}{\sqrt{n}} \Zd} \Big) \le \del.
\end{equation}
Note that the left side of \eqref{ineq:del2} decreases with increasing $n$.  Define $\del$ and $N$ as in \eqref{eq:del} and \eqref{eq:N}, respectively; then, inequality \eqref{ineq:del2} holds for all $d \ge 0$ and all $n > N$, which implies that $\mLn^{\Rn} \elln(x, m) \ge 0$ for all $(x, m) \in \mD$ and all $n > N$.  The conclusion in \eqref{eq:psio_scale}, then, follows from Theorem \ref{thm:sub}.
\end{proof}

In the following two propositions, we show that the condition in \eqref{eq:Zd} includes most of the usual claim distributions with light tails.

\begin{proposition}\label{prop:bounded}
Suppose $Y$ has bounded support in $\R_+$, then \eqref{eq:Zd} holds.
\end{proposition}

\begin{proof}
Let $b > 0$ be such that $\Fy(b) = 1$.  Let $a = \rhod/\sqrt{\varsigma}$, and define the function $G$ by
\begin{equation}\label{eq:G}
G(d) = \E\big(e^{a \Zd}\big),
\end{equation}
which implies
\[
G(d) = \dfrac{\int_d^b e^{a(y - d)} d\Fy(y)}{\Sy(d)} \, \mathds{1}_{\{ d < b\}}.
\]
Because $e^{a(y - d)} \le e^{a(b - d)}$ for all $0 \le y \le b$, we have
\[
G(d) \le e^{a(b - d)},
\]
so its supremum over $\R_+$ is finite.
\end{proof}

\begin{proposition}\label{prop:unbounded}
Suppose $Y$ has a probability density function $\fy$ with unbounded support on $\R_+$, and let $\hy$ denote the hazard rate function of $Y$, that is,
\begin{equation}
\hy(y) = \dfrac{\fy(y)}{\Sy(y)}, 
\end{equation}
for $y \ge 0$.  Furthermore, suppose 
\begin{equation}
\lim_{y \to \infty} \hy(y) = \ell.
\end{equation}
If $\ell = 0$, then $\My(y) = \infty$ for all $y > 0$, which contradicts our assumption concerning $Y$'s moment generating function. On the other hand, if $\ell > 0$, then \eqref{eq:Zd} holds.
\end{proposition}

\begin{proof}
The first conclusion in the statement of this proposition follows from Theorem 2.5.1 in Rolski et al.\ \cite{RSST1999}, so suppose $\ell > 0$.  Let $\varsigma$ be large enough so that $a = \rhod/\sqrt{\varsigma} < \ell/2$ with $\My(a) < \infty$.  As in the proof of Proposition \ref{prop:bounded}, define the function $G$ by \eqref{eq:G}.  By applying L'H\^opital's rule, we obtain
\begin{align*}
\lim_{d \to \infty} G(d) &\; = \lim_{d \to \infty} \dfrac{\int_d^\infty e^{ay} \fy(y) dy}{e^{ad} \Sy(d)} \\
&\overset {L\hbox{'}H} {=} \lim_{d \to \infty} \dfrac{- e^{ad} \fy(d)}{a e^{ad} \Sy(d) - e^{ad} \fy(d)} \\
&\; = \lim_{d \to \infty} \dfrac{\hy(d)}{-a + \hy(d)} = \dfrac{\ell}{\ell - a} > 0.
\end{align*}
Because $G$ is continuous on $\R_+$, it follows that $\sup_{d \ge 0} G(d) < \infty$.
\end{proof}

\subsection{Convergence of $\psin$ to $\psid$}\label{sec:psin_lim}

In this section we prove an important result, namely, that as $n \to \infty$, $\psin$ converges to $\psid$ uniformly on $\mD$, with rate of convergence of order $\mO(n^{-1/2})$.  In the following theorem, we combine the results of Propositions \ref{prop:psin_upper} and \ref{prop:psin_lower}.

\begin{theorem}\label{thm:psin_lim}
If \eqref{eq:Zd} holds,
then there exist $C' > 0$ and $N' > 0$ such that, for all $n \ge N'$ and $(x, m) \in \mD$,
\begin{align}\label{eq:psin_lim}
\big| \psin(x, m) - \psid(x, m) \big| \le \dfrac{C'}{\sqrtn} \, .
\end{align}
\end{theorem}

\begin{proof}
From Propositions \ref{prop:psin_upper} and \ref{prop:psin_lower} it follows that
\[
\left(1 - \dfrac{\del}{\sqrtn}\right) \psid(x, m) < \psin(x, m) < \un(x, m).
\]
Subtracting $\psid(x, m)$ from each side yields
\begin{align}\label{ineq:lim1}
- \, \dfrac{\del}{\sqrtn} \,  \psid(x, m) < \psin(x, m) -  \psid(x, m) < \un(x, m) -  \psid(x, m).
\end{align}
Clearly, the left side is bounded below by $-\del/\sqrtn$.  From \eqref{eq:psid}, \eqref{eq:hd}, \eqref{eq:un}, and \eqref{eq:kn}, we deduce that the right side is positive and equals
\begin{align}\label{ineq:lim2}
&\left( 1 - e^{-\rhod(1 - \alp)m} \right)^{\frac{\alp}{1 -\alp}} \left( 1 - e^{-\rhod(x - \alp m)} \right) - \left( 1 - e^{-(\rhod - C/\sqrtn)(1 - \alp)m} \right)^{\frac{\alp}{1 -\alp}} \left( 1 - e^{-(\rhod - C/\sqrtn)(x - \alp m)} \right) \notag \\
&= \left( 1 - e^{-\rhod(x - \alp m)} \right) \left[ \left( 1 - e^{-\rhod(1 - \alp)m} \right)^{\frac{\alp}{1 -\alp}} - \left( 1 - e^{-(\rhod - C/\sqrtn)(1 - \alp)m} \right)^{\frac{\alp}{1 -\alp}} \right] \notag \\
&\quad + \left( 1 - e^{-(\rhod - C/\sqrtn)(1 - \alp)m} \right)^{\frac{\alp}{1 -\alp}} \left[ \left( 1 - e^{-\rhod(x - \alp m)} \right) - \left( 1 - e^{-(\rhod - C/\sqrtn)(x - \alp m)} \right) \right].
\end{align}

First, analyze the next-to-the-last line in \eqref{ineq:lim2}:  for $n > (C/\rhod)^2$, we have
\begin{align*}
0 &\le \left( 1 - e^{-\rhod(x - \alp m)} \right) \left[ \left( 1 - e^{-\rhod(1 - \alp)m} \right)^{\frac{\alp}{1 -\alp}} - \left( 1 - e^{-(\rhod - C/\sqrtn)(1 - \alp)m} \right)^{\frac{\alp}{1 -\alp}} \right] \\
&\le \left[ \left( 1 - e^{-\rhod(1 - \alp)m} \right)^{\frac{\alp}{1 -\alp}} - \left( 1 - e^{-(\rhod - C/\sqrtn)(1 - \alp)m} \right)^{\frac{\alp}{1 -\alp}} \right] \\
&= \left( 1 - e^{-\rhod(1 - \alp)m} \right)^{\frac{\alp}{1 -\alp}} \left[ 1 - \left( \dfrac{1 - e^{-(\rhod - C/\sqrtn)(1 - \alp)m}}{1 - e^{-\rhod(1 - \alp)m}} \right)^{\frac{\alp}{1 -\alp}} \right] \\
&\le 1 - \left( \dfrac{1 - e^{-(\rhod - C/\sqrtn)(1 - \alp)m}}{1 - e^{-\rhod(1 - \alp)m}} \right)^{\frac{\alp}{1 -\alp}}.
\end{align*}
It is straightforward to show that the last line of the above expression decreases with respect to $m$; therefore, by applying L'H\^opital's rule to take the limit as $m$ goes to $0$ of the expression in parentheses, we obtain
\begin{align*}
0 &\le \left( 1 - e^{-\rhod(x - \alp m)} \right) \left[ \left( 1 - e^{-\rhod(1 - \alp)m} \right)^{\frac{\alp}{1 -\alp}} - \left( 1 - e^{-(\rhod - C/\sqrtn)(1 - \alp)m} \right)^{\frac{\alp}{1 -\alp}} \right] \\
&\le 1 - \left( 1 - \dfrac{C}{\rhod \sqrtn} \right)^{\frac{\alp}{1 - \alp}}.
\end{align*}
If $b \ge 1$, then $1 - (1 - x)^b$ is concave on $[0, 1]$, so lies below its tangent line at $x = 0$, which implies $1 - (1 - x)^b \le bx$ for all $0 \le x \le 1$.  If $0 < b < 1$, then $1 - (1 - x)^b$ is convex on $[0, 1]$, so lies below its secant line between $(0, 0)$ and $(1, 1)$, which implies $1 - (1 - x)^b \le x$ for all $0 \le x \le 1$.  Thus, by choosing any $M > (C/\rhod)^2$, we conclude that $n \ge M$ implies
\begin{align}\label{ineq:lim2_1}
0 &\le \left( 1 - e^{-\rhod(x - \alp m)} \right) \left[ \left( 1 - e^{-\rhod(1 - \alp)m} \right)^{\frac{\alp}{1 -\alp}} - \left( 1 - e^{-(\rhod - C/\sqrtn)(1 - \alp)m} \right)^{\frac{\alp}{1 -\alp}} \right] \notag \\
&\le \max\left( \frac{\alp}{1 - \alp}, \, 1 \right) \cdot \dfrac{C}{\rhod \sqrtn}.
\end{align}

Next, analyze the last line of \eqref{ineq:lim2}:  for $n > (C/\rhod)^2$, calculus shows that
\begin{equation}\label{ineq:lim3}
\left( 1 - e^{-\rhod(x - \alp m)} \right) - \left( 1 - e^{-(\rhod - C/\sqrtn)(x - \alp m)} \right) \le \left(1 - \frac{C}{\rhod \sqrtn} \right)^{\frac{\rhod \sqrtn}{C}} \left(\frac{\frac{C}{\sqrtn}}{{\rhod - \frac{C}{\sqrtn}}} \right).
\end{equation}
The first factor on the right side of \eqref{ineq:lim3} converges to $e^{-1}$; thus, there exists $M' > (C/\rhod)^2$ such that
\[
\left(1 - \frac{C}{\rhod \sqrtn} \right)^{\frac{\rhod \sqrtn}{C}} \le 2 e^{-1},
\]
for all $n \ge M'$, which implies that the last line in \eqref{ineq:lim2} satisfies
\begin{align}\label{ineq:lim4}
0 &\le \left( 1 - e^{-(\rhod - C/\sqrtn)(1 - \alp)m} \right)^{\frac{\alp}{1 -\alp}} \left[ \left( 1 - e^{-\rhod(x - \alp m)} \right) - \left( 1 - e^{-(\rhod - C/\sqrtn)(x - \alp m)} \right) \right] \notag \\
&\le 2 e^{-1} \, \dfrac{C}{\rhod - \frac{C}{\sqrt{M'}}} \cdot \dfrac{1}{\sqrtn},
\end{align}
for all $n \ge M'$.

By combining \eqref{ineq:lim2_1} and \eqref{ineq:lim4}, by setting $N' = \max(N, M, M')$, and by setting
\[
C' = \max\left\{ \del, \; \max\left( \frac{\alp}{1 - \alp}, \, 1 \right) \cdot \dfrac{C}{\rhod} + 2 e^{-1} \, \dfrac{C}{\rhod - \frac{C}{\sqrt{M'}}} \right\},
\]
we obtain \eqref{eq:psin_lim}.
\end{proof}


\subsection{$\mO(n^{-1/2})$-optimality of retaining $\Rdn$ in the $n$-scaled model}\label{sec:Rd_opt}

We end this paper by showing that, if an insurer follows the optimal retention strategy for the diffusion model but surplus follows the $n$-scaled model, then the resulting probability of drawdown is $\mO(n^{-1/2})$-optimal.  To that end, let $\psidn$ denote the probability of drawdown when the insurer retains $\Rdn$ in \eqref{eq:Rdn} when surplus follows the $n$-scaled model.  Then, we have the following theorem whose proof is similar to the proof of Theorem \ref{thm:super}, so we omit it.

\begin{theorem}\label{thm:super_Rdn}
Suppose $v \in \mC^{1,1}(\mD)$ is a bounded function that satisfies the following conditions:
\begin{enumerate}
\item[$(i)$] $\lim \limits_{m \to \infty} v(m, m) \ge 0$. 
\item[$(ii)$] $v(x, m)$ is defined for all $x \le m$ and $m \ge 0$, with $v(x, m) = 1$ for all $x < \alp m$.
\item[$(iii)$] $v_m(m, m) \le 0$ for all $m \ge 0$.
\item[$(iv)$] $\mLn^{\Rdn} v(x, m) \le 0$ for all $(x, m) \in \mD$.
\end{enumerate}

\noindent
Then, $\psidn \le v$ on $\mD$.  \qed
\end{theorem}

By analogy with the expressions for $\psid$ and $\opsin$ in \eqref{eq:psid} and \eqref{eq:opsin}, respectively, define $\opsidn$ for $x \le m$ and $m \ge 0$ as follows:
\begin{equation}\label{eq:opsidn}
\opsidn(x, m) =
\begin{cases}
1 - \hdn(m) \left( 1 - e^{-\rhon(\Rd)(x - \alp m)} \right), &\quad (x, m) \in \mD,\\
1, &\quad x < \alp m,
\end{cases}
\end{equation}
in which $\hdn$ is defined by
\begin{equation}\label{eq:hdn}
\hdn(m) = \left( 1 - e^{-\rhon(\Rd)(1 - \alp)m} \right)^{\frac{\alp}{1 -\alp}},
\end{equation}
for all $m \ge 0$.  In \eqref{eq:opsidn} and \eqref{eq:hdn}, $\rhon(\Rd) > 0$ uniquely solves \eqref{eq:rhon_HJB} with the $\sup_R$ removed, as defined at the beginning of the proof of Lemma \ref{lem:rhon_le_rhod}.

In the following lemma, we use Theorem \ref{thm:super_Rdn} to prove that $\opsidn$ is an upper bound of $\psidn$.

\begin{lemma}\label{lem:opsidn_upper}
For all $(x, m) \in \mD$,
\begin{equation}\label{eq:opsidn_upper}
\psidn(x, m) \le \opsidn(x, m).
\end{equation}
\end{lemma}

\begin{proof}
We prove this lemma via Theorem \ref{thm:super_Rdn}.  First, note that $\opsidn \in \mC^{1,1}(\mD)$ by its definition.  Next, we go through each condition in Theorem \ref{thm:super_Rdn} in turn.

\medskip

\noindent {\it Condition} (i):
\begin{align*}
\lim_{m \to \infty} \opsidn(m, m) &= \lim_{m \to \infty} \left( 1 - \hdn(m) \left( 1 - e^{-\rhon(\Rd)(1 - \alp)m} \right) \right) \\
&= 1 - \lim_{m \to \infty} \left( 1 - e^{-\rhon(\Rd)(1 - \alp)m} \right)^{\frac{\alp}{1 -\alp} + 1} = 1 - 1 = 0,
\end{align*}
so condition (i) is satisfied with equality.

\medskip

\noindent {\it Condition} (ii): This condition, namely, that $\opsidn(x, m)$ is defined for all $x \le m$ and $m \ge 0$, with $\opsidn(x, m) = 1$ for all $x < \alp m$, is satisfied by the definition of $\opsidn$ in \eqref{eq:opsidn}.

\medskip

\noindent {\it Condition} (iii):  For $m > 0$, differentiate $\opsidn$ with respect to $m$ and simplify the expression to obtain
\begin{align*}
(\opsidn)_m(x, m) = \dfrac{\alp \rhon(\Rd) \hdn(m)}{1 - e^{-\rhon(\Rd)(1 - \alp)m}} \left(e^{-\rhon(\Rd)(x - \alp m)} - e^{-\rhon(\Rd)(1 - \alp)m}  \right),
\end{align*}
which equals $0$ when $x = m$.  Therefore, condition (iii) is satisfied with equality.

\medskip

\noindent {\it Condition} (iv):  From \eqref{eq:lem32_0}, we know $\rhon(\Rd)$ satisfies
\[
- \kap + \la \left[ (\sqrtn + \tet) \E \Rd + \eta \E(Y \Rd) - \dfrac{\eta}{2} \, \E(\Rd^2) \right] = \dfrac{n\la}{\rhon(\Rd)} \left( \E e^{\rhon(\Rd)\Rd/\sqrtn} - 1 \right).
\]
Then, for $(x, m) \in \mD$, the expression in \eqref{eq:mLn} gives us
\begin{align*}
\mLn^{\Rdn} \opsidn(x, m) &= \left[- \kapn + \lan \left( (1 + \tetn) \En \Rdn + \eta \En(\Yn \Rdn) - \frac{\eta}{2} \, \En((\Rdn)^2) \right) \right]  (\opsidn)_x(x, m) \\
&\quad + \lan\big(  \En \opsidn(x - \Rdn, m) - \opsidn(x, m) \big) \\
&=  \left[- \kap + \la \left( (\sqrtn + \tet) \E\Rd + \eta \E(Y\Rd) - \frac{\eta}{2} \, \E(\Rd^2) \right) \right]  (\opsidn)_x(x, m) \\
&\quad + n\la \big(  \E\opsidn(x - \Rd/\sqrtn, m) - \opsidn(x, m) \big) \\
&= \dfrac{n\la}{\rhon(\Rd)} \left( \E e^{\rhon(\Rd)\Rd/\sqrtn} - 1 \right) \big(- \rhon(\Rd) \hdn(m) e^{-\rhon(\Rd)(x - \alp m)} \big) \\
&\quad + n\la \left[ \E \Big( \left(1 - \hdn(m) \left( 1 - e^{-\rhon(\Rd)(x - \Rd/\sqrtn - \alp m)} \right)\right) \mathds{1}_{\{ x - \Rd/\sqrtn \ge \alp m\}} \Big) \right] \\
&\quad + n\la \left[ \E \big( \mathds{1}_{\{ x - \Rd/\sqrtn < \alp m\}} \big) - \left(1 - \hdn(m) \left( 1 - e^{-\rhon(\Rd)(x - \alp m)} \right) \right) \right] \\
&= - n \la \hdn(m) \E \Big( \big( e^{\rhon(\Rd)(\alp m - (x - \Rd/\sqrtn))} - 1 \big) \mathds{1}_{\{ x - \Rd/\sqrtn < \alp m\}} \Big) \\
&\le 0.
\end{align*}
Thus, we have proved condition (iv).  Theorem \ref{thm:super_Rdn}, then, implies inequality \eqref{eq:opsidn_upper} on $\mD$.
\end{proof}

We obtain the following proposition from Lemma \ref{lem:opsidn_upper}, in which we show that $\un$ in \eqref{eq:un} is an upper bound of $\psidn$.

\begin{proposition}\label{prop:psidn_upper}
Let $C$ and $N$ be as in the statement of Lemma {\rm \ref{lem:rhod_le_rhon}}; then, for $n \ge N$,
\begin{equation}\label{eq:psidn_upper}
\psidn \le \un,
\end{equation}
on $\mD$, in which $\un$ is defined in \eqref{eq:un}.
\end{proposition}

\begin{proof}
It is straightforward to show that $\opsidn$ in \eqref{eq:opsidn} decreases with respect to $\rhon(\Rd) > 0$ on $\mD$.  Thus, if we replace $\rhon(\Rd)$ in $\opsidn$'s definition with a (positive) parameter less than $\rhon(\Rd)$, then we get a function that is an upper bound of $\opsidn$.  That is exactly how we defined $\un$ in \eqref{eq:un} because, from inequality \eqref{eq:rhod_le_rhon_Rd} in the proof of Lemma \ref{lem:rhod_le_rhon}, we know $ 0 < \rhod - C/\sqrtn < \rhon(\Rd)$ for all $n \ge N$; thus, we have, from Lemma \ref{lem:opsidn_upper},
\[
\psidn \le \opsidn \le \un,
\]
on $\mD$.
\end{proof}

The following theorem is the main result of this paper, and it fully justifies using the optimal retention function for the diffusion approximation in the classical Cram\'er-Lundberg model when minimizing the probability of drawdown.

\begin{theorem}\label{thm:Rd_opt}
Suppose \eqref{eq:Zd} holds, and let $C' > 0$ and $N' > 0$ be as in Theorem {\rm \ref{thm:psin_lim}}.  Then, for all $n \ge N'$ and $(x, m) \in \mD$,
\begin{equation}\label{eq:Rd_opt}
\big| \psin(x, m) - \psidn(x, m) \big| \le \dfrac{2C'}{\sqrtn}.
\end{equation}
\end{theorem}

\begin{proof}
From the suboptimality of using $\Rdn$, we have $\psin \le \psidn$ on $\mD$.  From Proposition \ref{prop:psidn_upper}, we have $\psidn \le \un$ for $n \ge N$.  Thus, from the proof of Theorem \ref{thm:psin_lim}, we deduce
\[
- \, \dfrac{C'}{\sqrtn} \le \psin - \psid \le \psidn - \psid \le \un - \psid \le \dfrac{C'}{\sqrtn},
\]
on $\mD$ for $n \ge N'$; recall $N' \ge N$.  Thus,
\[
\big| \psid - \psidn \big| \le \frac{C'}{\sqrtn}
\]
on $\mD$ for $n \ge N'$.  This inequality, together with \eqref{eq:psin_lim} and the triangle inequality, proves \eqref{eq:Rd_opt}.
\end{proof}

\end{document}